\documentclass[a4paper,12pt]{amsart}
\usepackage{amsfonts,amsmath,amssymb}
\usepackage{hyperref}
\usepackage{fullpage}
\usepackage{mathrsfs}
\usepackage{prelim2e}
\usepackage{color}

\newtheorem{theo}{Theorem}[section]
\newtheorem{lemma}[theo]{Lemma}
\newtheorem{coro}[theo]{Corollary}
\newtheorem{prop}[theo]{Proposition}

\theoremstyle{definition}
\newtheorem{defi}[theo]{Definition}
\theoremstyle{remark}
\newtheorem{rem}[theo]{Remark}
\newtheorem{example}[theo]{Example}

\def\<#1,#2>{\langle #1,#2\rangle}

\renewcommand{\geq}{\geqslant}
\renewcommand{\leq}{\leqslant}

\newcommand{\mM}{\mathbb{M}}
\newcommand{\R}{\mathbb{R}}

\newcommand{\hilbert}[2]{\omega(#1/#2)}
\def\myhilbert(#1/#2){\omega(#1/#2)}

\newcommand{\pP}{\mathcal{P}}

\newcommand{\unit}{\mathbf{e}}
\newcommand{\C}{\mathcal{C}}
\newcommand{\CC}{\mathbb{C}}
\newcommand{\cM}{\mathcal{M}}
\newcommand{\cP}{\mathcal{P}}

\newcommand{\cX}{\mathcal{X}}

\newcommand{\sym}{\operatorname{S}}

\newcommand{\End}{\operatorname{End}}

\newcommand{\trace}{\operatorname{trace}}

\newcommand{\extr}{\operatorname{extr}}
\newcommand{\Diam}{\operatorname{diam}}
\newcommand{\firstdef}[1]{{\em #1}}

\newcommand{\othernorm}[1]{\|#1\|}
\newcommand{\argmax}[1]{\underset{#1}{\operatorname{arg}\,\operatorname{max}}\;}
\title{
Dobrushin ergodicity coefficient for Markov operators on cones, and beyond
}
\author{St\'ephane Gaubert}
\address{INRIA and CMAP\\
         \'Ecole Polytechnique\\
          91128 Palaiseau C\'edex, France}
\email[]{Stephane.Gaubert@inria.fr}
\thanks{The authors were partially supported by the Gaspard Monge Optimization Programme (PGMO, FMJH), 2012-2013}
\author{Zheng QU}
\address{CMAP and INRIA\\
         \'Ecole Polytechnique\\
          91128 Palaiseau C\'edex, France}
\email[]{zheng.qu@polytechnique.edu}

\keywords{consensus operator, Hilbert's projective metric, Hopf oscillation seminorm, contraction rate, non linear consensus system, Dobrushin's ergodicity coefficient, quantum channel}
\subjclass[2010]{Primary 47H09; Secondary 47B60, 15B51, 58B20 ,37A30 }
\begin{document}

\maketitle

\begin{abstract}
The analysis of classical consensus algorithms
relies on contraction properties of adjoints of Markov operators,
with respect to Hilbert's
projective metric or to a related family of seminorms
(Hopf's oscillation or Hilbert's seminorm).
We generalize these properties to abstract consensus operators over normal cones, which include the unital completely positive maps (Kraus operators) arising in quantum information theory.
In particular, we show that the contraction rate of such operators, with respect to the Hopf oscillation seminorm,
is given by an analogue of Dobrushin's ergodicity coefficient. We derive
from this result a characterization of the contraction rate of
a non-linear 
 flow, with respect to Hopf's oscillation seminorm and to Hilbert's projective metric.

\end{abstract}

\section{Introduction}
\subsection{Motivation: from Birkhoff's theorem to consensus dynamics}
The {\em Hilbert projective metric} $d_H$ on the interior of a (closed, convex, and pointed) cone $\C$ in a Banach space $\cX$ can be defined by:
$$
d_H(x,y):=\log\inf\{\frac{\beta}{\alpha}:\; \alpha,\beta>0,\; \alpha x\leq y\leq \beta x\},
$$
where $\leq$ is the partial order induced by $\C$, so that $x\leq y$ if $y-x\in \C$. Birkhoff~\cite{birkhoff57} characterized the
contraction ratio with respect to $d_H$ of a linear map $T$ preserving the interior $\C^0$ of the cone $\C$, 
$$
\sup_{x,y\in \C^0}\frac{d_H(Tx,Ty)}{d_H(x,y)}=\tanh(\frac{\Diam T(\C^0)}{4}),\qquad \Diam T(\C^0):=\sup_{x,y\in \C^0} d_H(Tx,Ty)\enspace.
$$
This fundamental result, which implies that a linear map sending the cone $\C$ into its interior is a strict contraction in Hilbert's metric, can be used
to derive the 
  Perron-Frobenius theorem
from
the Banach contraction mapping theorem, see~\cite{Bushell73,Kohlberg82,EvesonNuss95} for more information.

Hilbert's projective metric is related to the following family of seminorms. To any point $\unit \in \C^0$ is associated the seminorm
$$
x\mapsto \hilbert{x}{\unit}:=\inf\{\beta-\alpha: \alpha \unit \leq x\leq \beta \unit\}
$$
which is sometimes called {\em Hopf's oscillation}~\cite{hopf,Bushell73} or {\em Hilbert's seminorm}~\cite{arxiv1}. 
Nussbaum~\cite{nussbaum94}  showed that $d_H$ is precisely the weak Finsler metric obtained when taking $\hilbert{\cdot}{\unit}$ to be the infinitesimal distance at point $\unit$. In other words,
\[
d_H(x,y) = \inf_{\gamma} \int_0^1 \hilbert{\dot{\gamma}(s)}{\gamma(s)} ds 
\]
where the infimum is taken over piecewise $C^1$ paths $\gamma:[0,1]\to \C^0$ 
such that $\gamma(0)=x$ and $\gamma(1)=y$. 
He deduced that the contraction ratio,
with respect to Hilbert's projective metric,
 of a non linear map $f
:\C^0\rightarrow \C^0$ that is positively homogeneous of degree $1$
(i.e.\ $f(\lambda x)=\lambda f(x)$ for all $\lambda>0$), 
can be expressed in terms of the Lipschitz constants of the linear maps $Df(x)$ with respect to a family of Hopf's oscillation seminorms: 
\begin{align}
\sup_{x,y\in U} \frac{d_H(f(x),f(y))}{d_H(x,y)}=\sup_{x\in U}\sup_{z\in \cX ,\; \hilbert{z}{x}\neq 0 } \frac{\hilbert{Df(x)z}{f(x)}}{\hilbert{z}{x}} \enspace .
\end{align}
Hence, to arrive at an explicit formula for the contraction
rate in Hilbert's projective metric of non-linear maps, a basic issue
is to determine the Lipschitz constant $\kappa(T,\unit)$ of linear map $T$ with respect to Hopf's oscillation seminorm, i.e.,
\begin{align}\label{e-contract}
 \kappa(T,\unit): = \sup_{z\in \cX,\; \hilbert{z}{\unit}\neq 0} \frac{\hilbert{T(z)}{T(\unit)}}{\hilbert{z}{\unit}} \enspace .
\end{align}




The problem of computing the contraction rate~\eqref{e-contract} 
also arises in the study of consensus algorithms.
A {\em consensus operator} is a linear map $T$ which preserves the positive cone
$\C$ and fixes a unit element $\unit\in \C^0$: $T(\unit)=\unit$. A discrete time consensus system can be
described by
\begin{align}\label{a-xk+1Tk}
x_{k+1}=T_{k+1}(x_k),\quad k\in \mathbb N,
\end{align}
where $T_1,T_2,\dots$ is a sequence of consensus operators.
This model includes in particular the case in which $\cX=\R^n$, $\C=\R^n_+$, $\unit=(1,\cdots,1)^\top$ and $T_k(x)=T(x):=Ax$, for all $k$, where $A$ is a stochastic matrix. This has been studied in the field of communication networks, control theory and parallel computation~\cite{Hirsch89,Bertsekas89,Boydrand06,Moreau05,Blondel05convergencein,OlshevskyTsitsiklis,AngeliBliman}. Consensus operators
also arise in non-linear potential theory~\cite{dellacherie}.
Other interesting consensus operators are the unital completely positive maps acting on the cone of positive semidefinite matrices, corresponding to quantum channel maps~\cite{sepulchre,ReebWolf2011}. The term {\em noncommutative consensus}
is coined in~\cite{sepulchre} for the corresponding class of dynamical systems.

The main concern of consensus theory is the convergence of the orbit $x_k$ to a 
{\em consensus state}, which is nothing but a scalar multiple of the unit element. 
When $\cX=\R^n$, $\C=\R^n_+$ and $\unit=(1,\dots,1)^\top$, a widely used Lyapunov function for the consensus dynamics, first considered by Tsitsiklis (see~\cite{Tsitsiklis86}), is the ``diameter'' of the state $x$ defined as
$$
\Delta(x)=\max_{1\leq i,j\leq n} (x_i-x_j),
$$ 
which is precisely Hopf's oscillation seminorm $\hilbert{x}{\unit}$. 
It turns out that the latter seminorm
can still be considered as a Lyapunov function for a consensus operator $T$,
with respect to an arbitrary cone. 
When $\C=\R_+^n$, it is well known that if the contraction ratio of $T$ with
respect to the Hopf oscillation seminorm is strictly less than one, and if $T_k= T$,
for all $k$, 
then, the orbits of the consensus dynamics converge exponentially
to a consensus state. We shall see here that the same remains true
in general (Theorem~\ref{th-ex-con}). 
For time-dependent consensus systems, a common approach is to bound
 the contraction ratio of every product of $p$
consecutive  operators $T_{i+p}\circ \dots \circ T_{i+1}$, $i=1,2,\dots$, for a fixed $p$,
see for example~\cite{Moreau05}. Moreover, if $\{T_k:k\geq 1\}$ is a stationary ergodic random process, then the almost sure 
convergence of the orbits of~\eqref{a-xk+1Tk} to a consensus state
can be deduced by showing that $\mathbb E[\log \othernorm{T_{1+p}\dots T_{1}}_H]<0$ for some $p>0$, see Bougerol~\cite{Bougerol93}.
Hence,
in consensus applications, a central issue is again to compute
the contraction ratio~\eqref{e-contract}. 

\subsection{Main results}
Our first result characterizes the contraction ratio~\eqref{e-contract},
in a slightly more general setting. We consider a bounded linear map $T$ from a Banach space $\cX_1$ to a Banach space $\cX_2$. The latter are equipped with
normal cones $\C_i\subset \cX_i$, and {\em unit} elements
$\unit_i \in \C_i^0$. 
\begin{theo}[Contraction rate in Hopf's oscillation seminorm]\label{th-intro1}
 Let $T:\cX_1 \to \cX_2$ 
be a bounded linear map such that $T(\unit_1)\in \mathbb{R} \unit_2$.
Then
\[
\sup_{\substack{z\in \cX_1\\ \hilbert{z}{\unit_1}\neq 0}}
\frac{\hilbert{T(z)}{\unit_2}}{\hilbert{z}{\unit_1}}
 = \frac 1 2 \sup_{\substack{\nu,\pi \in \operatorname{extr} \cP(\unit_2)\\ \nu\perp\pi}} 
\|T^\star(\nu)-T^\star(\pi)\|_T^\star=\sup_{\substack{\nu,\pi \in \operatorname{extr}\cP(\unit_2)\\ \nu\perp\pi}} \sup_{ x\in[0, \unit_1]}\<\nu-\pi, T(x)>.
\]
\end{theo}
The notation and notions used in this theorem are
detailed in Section~\ref{sec-operatornormch}.
In particular, we denote by 
the same symbol $\leq$ 
the order relations induced by the two cones $\C_i$, $i=1,2$;
$\mathcal{P}(\unit_2)=\{\mu\in \C_2^\star :\, \<\mu,\unit_2> =1\}$ denotes the
abstract {\em simplex} of the dual Banach space $\cX_2^\star$ of $\cX_2$,
where $\C_2^\star :=\{\mu\in \cX_2^\star:\, \<\mu,x> \geq 0 ,\;\forall x\in \C_2\}$ is the {\em dual cone} of $\C_2$;
$\extr$ denotes the extreme points of a set; $\bot$ denotes a certain
{\em disjointness} relation, which will be seen to generalize
the condition that two measures have disjoint supports;
$[u,v]:=\{x\in \cX_1:\, u\leq x\leq v\}$, for all $u,v\in \cX_1$,
and $T^\star$ denotes the adjoint of $T$. 
We shall make use of the following norm, which we call {\em Thompson's norm},
\[ \|z\|_{T} = \inf\{\alpha>0:\, -\alpha \unit_1\leq z\leq \alpha\unit_1\}\]
on the space $\cX_1$, and denote by $\|\cdot\|_T^\star$ the dual norm.

When $\C=\R_+^n$, and $T(z)=Az$ for some stochastic matrix $A$, we shall
see that the second supremum in Theorem~\ref{th-intro1} is simply
\[
 \frac 1 2 \max_{i<j} \sum_{1\leq k\leq n}|A_{ik}-A_{jk}| 
=  \frac 1 2 \max_{i<j} \|A_{i\cdot}-A_{j\cdot}\|_{\ell_1} \enspace ,
\]
where $A_{i\cdot}$ denotes the $i$th row of the matrix $A$.
This quantity is called {\em Doeblin contraction coefficient} in the theory of Markov chains; it is known to determine the contraction rate of the
adjoint $T^\star$ with respect to the $\ell_1$ (or total variation) metric,
see~\cite{PeresLevin}. Moreover, the last supremum in Theorem~\ref{th-intro1}
can be rewritten more explicitly as
\[
1-\min_{i<j}\sum_{s=1}^n \min(A_{is},A_{js}) \enspace,
\]
a term which is known as {\em Dobrushin's ergodicity coefficient}~\cite{Dobrushin56}. Note that in general, the norm $\|\cdot\|_T^\star$ can be thought of as an
abstract version of the $\ell_1$ or total variation norm.

When specializing to a unital completely positive map $T$ on the cone
of positive semidefinite matrices, representing a quantum channel~\cite{sepulchre,ReebWolf2011}, we shall see that the last supremum in Theorem~\ref{th-intro1} coincides with the following expression, which provides a non commutative analogue of Dobrushin's ergodicity coefficient (see Corollary~\ref{coro-noncommDobr}):
$$
1-
\displaystyle\min_{\substack{X=(x_1,\dots,x_n)\\XX^*=I_n}}\min_{\substack{u,v:u^*v=0\\u^*u=v^*v=1}} \sum_{i=1}^n \min\{u^*T(x_ix_i^*)u,v^*T(x_ix_i^*)v\}
$$

Theorem~\ref{th-intro1} shows in particular, when $\C=\R_+^n$,
that
the contraction rate of $T$ with respect to Hopf's oscillation seminorm
is the same as the contraction rate of $T^\star$ with respect
to the $\ell_1$ norm, given by the classical formulas of Doeblin
and Dobrushin.

Theorem~\ref{th-intro1} can be thought of as the dual of a result of Reeb, Kastoryano, and Wolf~\cite[Prop.~12]{ReebWolf2011}, who gave a
closely related formula, without the disjointness restriction
(and assuming that the dimension is finite) 
for the contraction rate of $T^\star$ with respect to a certain
``base norm'', which is the dual of Thompson's norm. 
Thus, Theorem~\ref{th-intro1} 
characterizes the contraction rate of $T$, whereas Proposition~12
of~\cite{ReebWolf2011} characterizes the contraction rate of the adjoint $T^\star$. We shall derive here the equality of both contraction rates
from general duality considerations, exploiting
the observation that Hopf's oscillation seminorm coincides with the quotient norm of Thompson's norm (Lemma~\ref{l-quonorm}). Then, we deduce Theorem~\ref{th-intro1}
from 
a characterization of the extreme points of the unit ball in the dual 
space of the quotient normed space (Theorem~\ref{p-carac-ext}).
The duality between both approaches is discussed more
precisely in
%
Remarks~\ref{rem-basenormdisnorm} and~\ref{rem-comReeb}.

Then, we derive analogous results for flows.
In particular, some consensus systems are driven by non linear ordinary differential equations~\cite{Strogatz00fromkuramoto,Saber03}:
$$
\dot x=\phi(x)
$$
where $\phi(x+\lambda \unit)=\phi(x)$ for all $\lambda \in \R$. 
The subclass of maps $\phi$ that yield an order preserving flow
is of interest in non-linear potential theory. In this context,
the opposite of the map $\phi$ has been called a 
{\em derivator} by Dellacherie~\cite{dellacherie}.

For such systems, it is interesting to consider the contraction 
rate of the flow with respect to Hopf's oscillation seminorm. In particular, if the contraction rate is negative, we deduce an exponential convergence 
of the orbits of the system to a consensus state. For simplicity, we only consider here a flow on finite dimensional space (then, a -closed, convex, and pointed- cone is automatically normal).
\begin{theo}[Contraction rate of flows with respect to Hopf's oscillation]\label{th-intro3}
The contraction rate $\alpha(U)$ with respect
to Hopf's oscillation seminorm, of the flow 
of the differential equation $\dot{x}=\phi(x)$, restricted to
a convex open subset $U\subset \cX$, is given by:
$$
\alpha(U)=\sup_{x\in U} h(D\phi(x))
$$
\end{theo}
Here $h(L)$ is defined to be the contraction rate in Hopf's oscillation seminorm of the {\em linear} differential equation $\dot x=L(x)$. 
It is given explicitly by
(Proposition~\ref{pr-alphaHilsemiL}):
$$
h(L):=-\inf_{\nu,\pi\in \extr\pP(\unit)}\inf_{\substack{x \in \extr([0, \unit]) \\ \<\nu,x>+\<\pi,\unit-x>=0}} \<\nu,L(x)>+\<\pi,L(\unit-x)>.
$$ 
For illustration we apply this result to some equations in $\R^n$. 

Our main results also include analogues of Theorem~\ref{th-intro1} and Theorem~\ref{th-intro3} concerning
the contraction rate in Hilbert's projective metric of a non-linear map
(Corollary~\ref{coro-Hilbertmetricnonlinearmap}), as well as a characterization of the contraction
rate in the same metric of a non-linear flow (Theorem~\ref{th-Hilbertmetricnonlinearflow}).

The paper is organized as follows. In Section~\ref{sec-pre} we give preliminary results on Thompson's metric, Hilbert's metric, and characterize
the extreme points of the dual unit ball. In Section~\ref{sec-operatornormch} we prove Theorem~\ref{th-intro1} and derive as corollary the analoguous result with 
respect to Hilbert's projective metric (Corollary~\ref{coro-Hilbertmetricnonlinearmap}).  In Section~\ref{sec-consensusoperator} we apply Theorem~\ref{th-intro1} to discrete time consensus operators. We determine the contraction rate of a linear flow in Hopf's oscillation seminorm in Section~\ref{sec-linearequation}. In Section~\ref{sec-nonlinearconsensus} we show Theorem~\ref{th-intro3} and
discuss some applications to non-linear consensus dynamics. In Section~\ref{sec-rateflowHilbertsmetric} we prove the analogue of Theorem~\ref{th-intro3} with respect to Hilbert's projective metric and show its applications.

\section{Preliminaries}\label{sec-pre}

\subsection{Thompson's norm and Hopf's oscillation seminorm }

We consider a real Banach space $(\cX, \othernorm{\cdot})$ and its dual space $\cX^*$. 
Let $\C\subset \cX$ be a closed pointed convex cone with non empty interior $\C^ 0$, i.e., $\alpha \C\subset \C$ for $\alpha \in \R^{+}$, $\C+\C\subset \C$ and $\C\cap (-\C)=0$.
 We define the partial order $\leq $ 
induced by $\C$ on $\cX$ by
\[
x\leq y \Leftrightarrow y-x \in \C  .
\]
The {\em dual cone} of $\C$ is:
$$
\C^\star:=\{z \in \cX^\star| \<z,x>\geq 0, \enspace \forall x\in \C\}.
$$
Since $\C$ is a closed convex cone, it follows from the strong separation theorem that 
\begin{align}\label{a-xinCiff}
x\in \C \Leftrightarrow \<z,x>\geq 0,\enspace \forall z\in \C^\star.
\end{align}
For all $x\leq y$ we define the {\em order interval}:
$$
[x,y]:=\{z\in \cX|x\leq z\leq y\}.
$$
For $x\in \cX$ and $y\in \C^ 0$, following~\cite{nussbaum88}, we define
\begin{align}\label{a-eq4}
\begin{array}{l}
M(x/y):=\inf\{t\in \R:x\leq ty\},\\
m(x/y):=\sup\{t\in \R:x\geq ty\} .
\end{array}
\end{align}
Observe
that since $y\in\C^0$, and since $\mathcal{C}$ is closed
and pointed, the two sets in~\eqref{a-eq4} are non-empty, closed, and bounded
from below and from above, respectively. In particular, $m$ and $M$ take
finite values.
The difference between $M$ and $m$ is called \firstdef{oscillation}~\cite{Bushell73}:
$$
\hilbert{x}{y}:=M(x/y)-m(x/y).
$$

Let $\unit$ 
denote a distinguished element in the interior of $\C$, which we shall
call a {\em unit}. We define
\[
\|x\|_{T}:=\max( M(x/\unit), -m(x/\unit))
\]
which we call {\em Thompson's norm}, with respect to the element $\unit$, and
\[
\|x\|_H:= \hilbert{x}{\unit}
\]
which we call Hopf's oscillation seminorm with respect to the element $\unit$. 

We assume that the cone is {\em normal}, meaning that there exists a constant $K>0$ such that
$$
0\leq x\leq y\Rightarrow \othernorm{x}\leq K\othernorm{y}.
$$
It is known that under this assumption the two norms $\othernorm{\cdot}$ and $\othernorm{\cdot}_T$ are equivalent, see~\cite{nussbaum94}. 
Therefore the space $\cX$ equipped with the norm $\othernorm{\cdot}_T$ is a Banach space. Since Thompson's norm $\othernorm{\cdot}_T$ is always defined
with respect to a particular element, we write $(\cX,\unit,\othernorm{\cdot}_T)$ instead of $(\cX,\othernorm{\cdot}_T)$.

By the definition and~\eqref{a-xinCiff}, Thompson's norm with respect to $\unit$ can be calculated by:
\begin{align}\label{e-carac-f}
\|x\|_T = \sup_{z \in \C^\star} \frac{|\<z,x>|}{\<z,\unit>}.
\end{align}

\begin{example}\label{rem1}
 We consider the space $\cX=\R^n$, the closed convex cone $\C=\R^n_+$ and the unit element $\unit=\bold 1:=(1,\dots,1)^T$. 
It can be checked that 
 Thompson's norm with respect to $\bold 1$ is nothing but the sup norm
$$
\othernorm{x}_T=\max_i |x_i|=\othernorm{x}_{\infty},$$
whereas Hopf's oscillation seminorm with respect to $\bold 1$ is the so called
{\em diameter}:
$$
\othernorm{x}_H=\max_{1\leq i,j\leq n} (x_i-x_j)
=\Delta(x).
$$
\end{example}

\begin{example}\label{ex-Symn}
 Let $\cX=\sym_n$, the space of Hermitian matrices of dimension $n$ and $\C=\sym_n^+\subset \sym_n$, the cone of positive semi-definite matrices. Consider the unit element $\unit=I_n$, the identity matrix of dimension $n$. Then Thompson's norm with respect to $I_n$ is nothing but the sup norm of the spectrum
of $X$, i.e.,
$$
\othernorm{X}_T=\max_{1\leq i\leq n} \lambda_i(X)=\othernorm{\lambda(X)}_{\infty},
$$
where $\lambda(X):=(\lambda_1(X),\dots,\lambda_n(X))$, $\lambda_1(X)\leq \dots\leq \lambda_n(X)$, is the vector of ordered eigenvalues of $X$, counted
with multiplicities,
whereas Hopf's oscillation seminorm with respect to $I_n$ is the diameter of the spectrum:
$$
\othernorm{X}_H=\max_{1\leq i,j\leq n} (\lambda_i(X)-\lambda_j(X))=\Delta(\lambda(X)).
$$
\end{example}


\subsection{Simplex in the dual space and dual unit ball}
We denote by $(\cX^\star,\unit,\othernorm{\cdot}_T^\star)$ the dual normed space of $(\cX,\unit,\othernorm{\cdot}_T)$ where
 the dual norm $\othernorm{\cdot}_T^\star$ of a continuous linear functional $z \in \cX^\star$ is defined by:
$$\othernorm{z}_T^\star:=\sup_{\othernorm{x}_T=1} \<z,x>.$$
We define:
\begin{align}\label{a-simp}\cP(\unit):=\{\mu \in \C^\star\mid \<\mu, \unit> =1\}\end{align}
the {\em simplex} with respect to $\unit$ of the dual Banach space $(\cX^\star,\unit,\|\cdot\|_T^\star)$.
\begin{rem}\label{rem-rn1}
When $\cX=\R^n$, $\C=\R^n_+$ and $\unit=\bold 1$ (Example~\ref{rem1}), the dual space $\cX^\star$ is $\cX=\R^n$ itself and the dual norm $\othernorm{\cdot}_T^\star$ is the $\ell_1$ norm:
$$
\othernorm{x}_T^\star=\sum_i |x_i|=\othernorm{x}_1.
$$ 
The simplex $\cP(\bold 1)$ defined in~\eqref{a-simp} is the simplex in $\R^n$ in the usual sense:
$$
\cP(\bold 1)=\{\nu\in \R^n_+: \sum_i \nu_i=1\},
$$
i.e., the set of probability measures on the discrete space $\{1,\dots,n\}$.
\end{rem}
\begin{rem}\label{rem-sym1}
 In the case of $\cX=\sym_n$, $\C=\sym_n^+$ and $\unit=I_n$ (Example~\ref{ex-Symn}), the dual space $\cX^\star$ is $\cX=\sym_n$ itself and the dual norm $\othernorm{\cdot}_T^\star$ is the trace norm:
$$
\othernorm{X}_T^\star=\sum_{1\leq i\leq n} |\lambda_i(X)|=\othernorm{X}_1,\quad X\in \sym_n
$$ 
The simplex $\cP(I_n)$ defined in~\eqref{a-simp} is the set of positive semi-definite matrices with trace $1$:
$$
\cP(I_n)=\{\rho\in \sym^n_+: \trace(\rho)=1\}.
$$
The elements of this set are called {\em density matrices} in quantum
physics, in which they are thought of as noncommutative analogues
of probabilities measure.
\end{rem}

The next lemma relates $\cP(\unit)$ and the unit ball $B_T^\star(\unit)$ of the space $(\cX^\star,\unit,\othernorm{\cdot}_T^\star)$. We denote by $\operatorname{conv}(S)$ the convex hull of a set $S$.

\begin{lemma}\label{l-BTstar}
The unit ball $B_T^\star(\unit)$
of the space $(\cX^\star,\unit,\othernorm{\cdot}_T^\star)$, satisfies
\begin{align}
B_T^\star(\unit) = \operatorname{conv}(\cP(\unit)\cup -\cP(\unit))
\label{e-bstar}
\end{align}
\end{lemma}
\begin{proof}
For simplicity we write $\pP$ instead of $\pP(\unit)$ and $B_T^ \star$ instead of $B_T^ \star(\unit)$ in the proof.
It follows from~\eqref{e-carac-f} that
\begin{align}\label{e-carac-norm}
\|x\|_T = \sup_{\mu \in \cP} |\<\mu,x>| = \sup_{\mu \in \cP\cup -\cP} \<\mu,x> \enspace .
\end{align}
Hence $\|z\|_T^\star\leq 1$
if and only if, for all $x\in \cX$,
\begin{align}
 \<z,x> \leq \|x\|_T = \sup_{\mu \in \cP\cup -\cP} \<\mu,x> \enspace .
\label{e-elt}
\end{align}
By the strong separation theorem~\cite[Thm 3.18]{MR1831176}, if $z$ did
not belong to the closed convex hull $\overline{\operatorname{conv}}(\cP\cup -\cP)$, the closure being understood in the weak star topology of $\cX^\star$,
there would exist a vector $x\in \cX$ and a scalar $\gamma$
such that $\<z,x> > \gamma \geq \<\mu,x> $, for all $\mu\in \cP\cup -\cP$,
contradicting~\eqref{e-elt}.
\[
B_T^\star = \overline{\operatorname{conv}}(\cP\cup -\cP)
\enspace .
\]
We claim that the latter closure operation can be dispensed
with. Indeed, by the Banach Alaoglu theorem, $B_T^\star$ is weak-star compact.
Hence, its subset $\cP$, which is weak-star closed,
is also weak-star compact. \typeout{proof in appendix not needed for ECC}
If $\mu\in B_T^\star$, by the characterization
of $B_T^\star$ above, $\mu$ is a limit, in the weak star topology,
of a net $\mu_a = s_a \nu_a -t_a \pi_a$ with
$s_a+t_a=1$, $s_a,t_a\geq 0$ and $\nu_a,\pi_a \in \cP$ for $a\in \mathcal A$.
By passing to a subnet we can assume that $s_a,t_a:a\in \mathcal A$ converge respectively to $s,t \in [0,1]$ such that $s+t=1$ and $\nu_a,\pi_a:a\in \mathcal A$ converge respectively in the weak-star topology to $\nu,\pi \in \cP$. 
It follows 
that $\mu = s\nu-t\pi \in \operatorname{conv}(\cP\cup -\cP)$.
\end{proof}
\begin{rem}\label{rem-basenormdisnorm}
We make a comparison with the paper~\cite{ReebWolf2011}. 
In a finite dimensional setting, Reeb, Kastoryano, and Wolf defined a \firstdef{base} $\mathcal{B}$ of a proper cone $\mathcal{K}$ in a vector space $\mathcal{V}$ to be 
a cross section of this cone, i.e.,
they take $\mathcal{B}$ to be the intersection of the cone $\mathcal{K}$ with a hyperplane given by a linear functional in the interior of the dual cone. So,
$\mathcal{V}$ corresponds to $\mathcal{X}^\star$ here, and, since $\mathcal{V}$ is of finite dimension, we can identify the dual of $\mathcal{V}$ 
to $\mathcal{X}$, and consider the dual cone $\mathcal{C}\simeq \mathcal{K}^\star\subset \mathcal{V}^\star \simeq  \mathcal{X}$.
Modulo this identification, the base $\mathcal{B}$ can be written precisely as $\mathcal{B}=\{\mu\in \mathcal{K}:\, \<\mu,\unit >=1\}$ for some $\unit$ in the interior of $\mathcal{K}^\star$, so that the base $\mathcal{B}$ coincides with the simplex $\mathcal{P}(\unit)$ considered here. They defined the \firstdef{base norm} of $\mu\in \mathcal{ V} $ with respect to $\mathcal{B}$ by:
$$
\othernorm{\mu}_{\mathcal{B}}=\inf\{\lambda\geq 0| \mu\in \lambda \operatorname{conv}(\mathcal{B}\cup-\mathcal{B})\}.
$$
Lemma~\ref{l-BTstar} shows that the base norm coincides with the dual norm of 
Thompson's norm: for $\nu\in \cX^{\star}$,
$$
\othernorm{\nu}_{\mathcal{B}}=\inf\{\lambda \geq 0|\nu \in \lambda \operatorname{conv}(\cP(\unit)\cup -\cP(\unit))\}
=\inf\{\lambda \geq 0|\nu \in \lambda B_T^\star\}=\othernorm{\nu}_T^\star.
$$
The set
$$
\tilde M=\{x\in \mathcal{V}^\star|0\leq x\leq \unit\}.
$$
is also considered in~\cite{ReebWolf2011}; leading to define
the \firstdef{distinguishability norm} of $\mu\in \mathcal{V}$ by:
\begin{align}\label{a-vxunit}
\othernorm{\mu}_{\tilde M}=\sup_{0\leq x\leq \unit} \<\mu,2x-\unit>.
\end{align}
It is shown there that
$$
\othernorm{\mu}_{\tilde M}=\othernorm{\mu}_{\mathcal{B}} \enspace .
$$

\end{rem}

\subsection{Extreme points of the dual unit ball}
We first show that Hopf's oscillation seminorm coincides with the norm 
on the quotient Banach space of $(\cX, \unit,2\othernorm{\cdot}_T)$ by the closed subspace $\R\unit$. 
\begin{lemma}\label{l-quonorm}
For all $x\in \cX$, we have:
\[
\|x\|_H = 2 \inf_{\lambda\in \R}\|x+\lambda \unit\|_T 
\]
\end{lemma}
\begin{proof}
$\|x+\lambda \unit\|_T = (M(x/\unit)+\lambda)\vee 
(-m(x/\unit)-\lambda)$ is minimal when
$(M(x/\unit)+\lambda)=
(-m(x/\unit)-\lambda)$. Substituting the value of $\lambda$ obtained
in this way in $\|x+\lambda \unit\|_T$, we arrive at the announced formula.
\end{proof}
A standard result~\cite[P.88]{Conway90FA} of functional analysis shows that if
$\mathcal{W}$ is a closed subspace of a Banach space $(\mathcal{X},\othernorm{\cdot})$, then the quotient space $\cX/\mathcal{W}$ is complete.
Besides, the dual of the quotient space $\mathcal{X}/\mathcal{W}$ can be identified
isometrically to the space of continuous linear forms on $\mathcal{X}$ that vanish on $\mathcal{W}$, equipped with the dual norm $\othernorm{\cdot}^\star$ of $\mathcal{X}^\star$. 
Specializing this result to $\mathcal{W}=\R\unit$, we get:
\begin{lemma}\label{l-norm-rel}
The quotient normed space $(\cX/\R\unit, \othernorm{\cdot}_H)$ is a Banach space. Its dual is
 $(\cM(\unit), \othernorm{\cdot}_H^\star)$ where 
$$
\cM(\unit):=\{\mu \in \cX^\star | \<\mu, \unit>=0\},
$$
and
\begin{align}\label{a-hs-ts}
\othernorm{\mu}_H^\star:=\frac{1}{2} \othernorm{\mu}_T^\star, \enspace \forall \mu \in \cM(\unit).
\end{align}
\end{lemma}

\if
The Hopf's oscillation seminorm defines a norm in the quotient space $\cX/\R\unit$. 
It is easily seen from the definitions that
\begin{align}\label{a-nHT}
\othernorm{x}_H\leq 2\othernorm{x}_T.
\end{align}
Thus we know that the dual space of $(\cX/\R\unit, \othernorm{\cdot}_H)$ is a subset of $\cX^\star$.  It is immediate that the dual space should be contained in the hyperplane:
\[
\cM = \{\mu\in \cX^\star\mid \<\mu,\unit> =0\}.
\]
The next lemma shows that the dual space is exactly the hyperplane.
\begin{lemma}\label{l-norm-rel}
The dual space of $(\cX/\R\unit, \othernorm{\cdot}_H)$ is $(\cM, \othernorm{\cdot}_H^\star)$ where 
\begin{align}\label{a-hs-ts}
\othernorm{\mu}_H^\star:=\frac{1}{2} \othernorm{\mu}_T^\star, \enspace \forall \mu \in \cM.
\end{align}
\end{lemma}
\begin{proof}
Let any $x\in \cX/\R \unit$ and 
$$
\lambda=-\frac{M(x/\unit)+m(x/\unit)}{2},
$$
then 
$$
\|x+\lambda \unit\|_T=\frac{1}{2}\othernorm{x}_H.
$$
Let any $\mu \in \cM$, then
\[
\<\mu,x> =\<\mu,x+\lambda \unit>\leq \|\mu\|_T^\star\|x+\lambda \unit\|_T= \frac{1}{2}\|\mu\|_T^\star \|x\|_H.
\]
Hence the hyperplane $\cM$ is the dual space of $(\cX/\R\unit, \othernorm{\cdot}_H)$ and the dual norm $\othernorm{\cdot}_H^\star$ should satisfy:
$$\|\mu\|_H^\star \leq \frac 1 2 \|\mu\|_T^\star.$$
Finally it follows from the relation~\eqref{a-nHT} that:
$$
\|\mu\|_H^\star \geq \frac 1 2 \|\mu\|_T^\star.
$$
\end{proof}
\fi
The above lemma implies that the unit ball  of the space $(\cM(\unit), \othernorm{\cdot}_H^\star)$, denoted by
$B_H^\star(\unit)$, satisfies:
\begin{align}\label{a-BHstar}
B_H^\star(\unit)=2B_T^\star(\unit) \cap \cM(\unit).
\end{align}

\begin{rem}
 In the case of $\cX=\R^n$, $\C=\R^n_+$ and $\unit=\bold 1$ (Example~\ref{rem1} and Remark~\ref{rem-rn1}),
Lemma~\ref{l-norm-rel} implies that for any two probability measures $\mu,\nu\in \pP(\bold 1)$, the dual norm $\othernorm{\mu-\nu}_H^\star$ is the total variation distance between $\mu$ and $\nu$:
$$
\othernorm{\mu-\nu}_H^\star=\frac{1}{2}\othernorm{\mu-\nu}_1=\othernorm{\mu-\nu}_{TV}
$$
\end{rem}

Before giving a representation of the extreme points of $B_H^\star(\unit)$,
 we define the disjointness relation $\bot$ on $\pP(\unit)$.
\begin{defi}
 For all $\nu,\pi \in \pP(\unit)$, we say that $\nu$ and $\pi$ are {\em disjoint}, denoted by $\nu \perp \pi$, if 
$$
\mu =\frac{\nu+\pi}{2}
$$
for all $\mu \in \pP(\unit)$ such that $\mu \geq \frac{\nu}{2}$ and $\mu\geq \frac{\pi}{2}$. 
\end{defi}
In particular, we remark the following property:
\begin{lemma}\label{l-nuperppi}
 Let $\nu,\pi\in \pP(\unit)$. The following assertions are equivalent:
\begin{itemize}
\item [(a)]$\nu\perp\pi$.
\item[(b)] The only elements $\rho,\sigma\in \pP(\unit)$ such that
$$
\nu-\pi=\rho-\sigma
$$
are $\rho=\nu$ and $\sigma=\pi$.
\end{itemize}
\end{lemma}
\begin{proof}
(a)$\Rightarrow$ (b): Let any $\rho,\sigma\in \pP(\unit)$ such that
$$
\nu-\pi=\rho-\sigma.
$$
Then it is immediate that 
$$
\nu+\sigma=\pi+\rho. 
$$
Let $\mu=\frac{\nu+\sigma}{2}=\frac{\pi+\rho}{2}$. Then $\mu\in\pP(\unit)$, $\mu\geq \frac{\nu}{2}$ and $\mu\geq \frac{\pi}{2}$. Since we assumed $\nu \perp\pi$, we obtain that $\mu=\frac{\nu+\pi}{2}$. It follows that $\rho=\nu$ and $\sigma=\pi$.

(b)$\Rightarrow$ (a): Let any $\mu\in \pP(\unit)$ such that $\mu\geq \frac{\nu}{2}$ and $\mu\geq \frac{\pi}{2}$. Then 
$$
\nu-\pi=(2\mu-\pi)-(2\mu-\nu).
$$
From (b) we know that $2\mu-\pi=\nu$.
\end{proof}

We denote by $\extr(\cdot)$ the set of extreme points of a convex set.
\begin{theo}\label{p-carac-ext}
The set of extreme points of $B_H^\star(\unit)$, denoted by $\extr{B_H^\star(\unit)}$, is characterized by:
$$
\extr{B_H^\star(\unit)}=\{\nu-\pi:\nu,\pi\in \extr{\pP(\unit)},\nu \perp \pi\}.
$$ 
\end{theo}
\begin{proof}
It follows from~\eqref{e-bstar} that every
point $\mu\in B_T^\star(\unit)$ can be written
as $\mu= s\nu - t\pi$ with $ s+t=1,s,t\geq 0$,
$\nu,\pi \in \cP$. Moreover, if $\mu\in \cM(\unit)$, 
$0=\<\mu,\unit>=s\<\nu,\unit>-t\<\pi,\unit>
= s-t$, and so $s=t=\frac{1}{2}$. 
Thus every $\mu \in B_T^\star(\unit) \cap \cM(\unit)$ can be written as
\[
\mu = \frac{\nu-\pi}{2}, \enspace \nu,\pi \in \pP(\unit).
\]
Therefore by~\eqref{a-BHstar} we proved that
\begin{align}\label{a-BHstarunit}
B_H^\star(\unit)=\{\nu-\pi:\nu,\pi\in \pP(\unit)\}.
\end{align}
Now let $\nu,\pi \in \extr \pP(\unit)$ and $\nu\perp \pi$. We are going to prove that $\nu-\pi\in \extr B_H^\star(\unit)$. Let $\nu_1,\pi_1,\nu_2,\pi_2\in \pP(\unit)$ such that
$$
\nu-\pi=\frac{\nu_1-\pi_1}{2}+\frac{\nu_2-\pi_2}{2}.
$$
Then 
$$
\nu-\pi=\frac{\nu_1+\nu_2}{2}-\frac{\pi_1+\pi_2}{2}.
$$
By Lemma~\ref{l-nuperppi}, the only possibility is $2\nu={\nu_1+\nu_2}$ and $2\pi=\pi_1+\pi_2$.
Since $\nu,\pi\in\extr \pP(\unit)$ we obtain that $\nu_1=\nu_2=\nu$ and $\pi_1=\pi_2=\pi$. Therefore $\nu-\pi\in \extr{B_H^\star(\unit)}$.

Now let $ \nu,\pi\in \pP(\unit)$ such that $\nu-\pi \in \extr{B_H^\star(\unit)}$. 
Assume by contradiction that $\nu$ is not extreme in $\pP(\unit)$ (the case in which
$\pi$ is not extreme can be dealt with similarly). Then, we can
find $\nu_1,\nu_2\in \pP(\unit)$, $\nu_1\neq \nu_2$, such that
$\nu = \frac{\nu_1+ \nu_2}{2}$. It follows that $$\mu = \frac{\nu_1-\pi}{2}
+\frac{\nu_2-\pi}{2},$$ where ${\nu_1-\pi},\nu_2-\pi$ are distinct
elements of $ B_H^\star(\unit)$, which is a contradiction. Next we show that $\nu\perp \pi$. To this end,
let any $\rho,\sigma\in \pP(\unit)$ such that 
$$
\nu-\pi=\rho-\sigma.
$$
Then 
$$
\nu-\pi=\frac{\nu-\pi+\rho-\sigma}{2}=\frac{\nu-\sigma}{2}+\frac{\rho-\pi}{2}.
$$
If $\sigma\neq \pi$, then $\nu-\sigma\neq \nu-\pi$ and this contradicts the fact that $\nu-\pi$ is extremal. Therefore $\sigma=\pi$ and $\rho=\nu$.
From Lemma~\ref{l-nuperppi}, we deduce that $\nu\perp\pi$.


\end{proof}

\begin{rem}
 When $\cX=\R^n$, $\C=\R^n_+$ and $\unit=\bold 1$ (Example~\ref{rem1} and Remark~\ref{rem-rn1}), the set of extreme points of $\pP(\bold 1)$ is the set of 
standard basis vectors $\{e_i\}_{i=1,\dots,n}$. The extreme points are pairwise disjoint.
\end{rem}
\begin{rem}\label{rem-xx*yy*}
When $\cX=\sym_n$, $\C=\sym_n^+$ and $\unit=I_n$ (Example~\ref{ex-Symn} and Remark~\ref{rem-sym1}), the set of extreme points of $\pP(I_n)$ is:
$$
\extr \pP(I_n)=\{xx^*: x\in \CC^n, x^*x=1\}.
$$
Two extreme points $xx^*$ and $yy^*$ are disjoint if and only if $x^*y=0$. To see this, note that if $x^*y=0$ then any Hermitian matrix $X$ such that $X\geq xx^*$ and $X\geq yy^*$ should satisfy
$X\geq xx^ *+yy^ *$. Hence by definition $xx^*$ and $yy^ *$ are disjoint. Inversely, suppose that $xx^ *$ and $yy^ *$ are disjoint and consider the spectral decomposition of the matrix $xx^*-yy^*$, i.e., there is $\lambda\leq 1$ and two orthonormal vectors $u,v$ such that $xx^*-yy^*=\lambda(uu^ *-vv^ *)$. It follows that $xx^*-yy^*=uu^ *-((1-\lambda)uu^ *+\lambda vv^ *)$.
By Lemma~\ref{l-nuperppi}, the only possibility is $yy^ *=(1-\lambda)uu^ *+\lambda vv^ *$ and $xx^ *=uu^*$ thus $\lambda=1$, $u=x$ and $v=y$. Therefore $x^ *y=0$.
\end{rem}

\section{The operator norm induced by Hopf's oscillation}\label{sec-operatornormch}
Consider two real Banach spaces $\cX_1$ and $\cX_2$. Let $\C_1\subset \cX_1$ and $\C_2\subset \cX_2$ be respectively 
two closed pointed convex normal cones with non empty interiors $\C_1^ 0$ and $\C_2^0$.
Let $\unit_1\in \C_1^ 0$ and $\unit_2\in \C_2^ 0$. 
Then, we know from Section~\ref{sec-pre} that the two quotient spaces $(\cX_1/\R \unit_1,\othernorm{\cdot}_H)$ and $(\cX_2/\R \unit_2,\othernorm{\cdot}_H)$ equipped with the Hopf's oscillation seminorms associated respectively to $\unit_1$
and $\unit_2$ are Banach spaces. The dual spaces of $(\cX_1/\R \unit_1,\othernorm{\cdot}_H)$ and $(\cX_2/\R \unit_2,\othernorm{\cdot}_H)$ are respectively the spaces $(\cM(\unit_1),\othernorm{\cdot}_H^ \star)$ and $(\cM(\unit_2),\othernorm{\cdot}_H^ \star)$ (see Lemma~\ref{l-norm-rel}).

Let $T$ denote a continuous
linear map from $(\cX_1/\R \unit_1,\othernorm{\cdot}_H)$ to $(\cX_2/\R \unit_2,\othernorm{\cdot}_H)$.
The operator norm of $T$, denoted by $\othernorm{T}_H$, is given by:
\[
\|T\|_H:= \sup_{x\in B_H(\unit_1)} \|T(x)\|_H 
\]
 The \firstdef{adjoint operator} $T^\star: (\cM(\unit_2),\othernorm{\cdot}_H^ \star) \rightarrow (\cM(\unit_1),\othernorm{\cdot}_H^ \star) $ of $T$ is by definition:
$$
\<T^\star(\mu),x>=\<\mu,T(x)>,\enspace \forall \mu \in \cM(\unit_2),x\in \cX_1.
$$
The operator norm  of $T^\star$, denoted by $\othernorm{T^ \star}_H^ \star$, is then:
\[ \|T^\star\|_H^\star:= \sup_{\mu \in B_H^\star(\unit_2)} \|T^\star(\mu)\|^\star_H.
\]

A classical duality result (see \cite[\S~6.8]{aliprantis}) shows that an operator and its adjoint have the same operator norm. In particular, 
$$
\|T\|_H = \|T^\star\|_H^\star.
$$
\begin{theo}\label{th-opnorm}
Let $T:\cX_1 \to \cX_2$ 
be a bounded linear map such that $T(\unit_1)\in \mathbb{R} \unit_2$. Then,
\[
\|T\|_H  = \frac 1 2 \sup_{\substack{\nu,\pi \in  \pP(\unit_2)}} 
\|T^\star(\nu)-T^\star(\pi)\|_T^\star=\sup_{\substack{\nu,\pi \in \pP(\unit_2)}} \sup_{ x\in[0, \unit_1]}\<\nu-\pi,T(x)>.
\]
Moreover, the supremum 
can be restricted to the set 
of extreme points:
\begin{align}\label{e-carac-TH}
\|T\|_H  = \frac 1 2 \sup_{\substack{\nu,\pi \in \operatorname{extr} \cP(\unit_2)\\ \nu\perp\pi}} 
\|T^\star(\nu)-T^\star(\pi)\|_T^\star=\sup_{\substack{\nu,\pi \in \operatorname{extr}\cP(\unit_2)\\ \nu\perp\pi}} \sup_{ x\in[0, \unit_1]}\<\nu-\pi, T(x)>.
\end{align}
\end{theo}
\begin{proof}
 We already noted that $\|T\|_H = \|T^\star\|_H^\star$. 
Moreover, 
\[
\|T^\star\|_H^\star
= \sup_{\mu\in B_H^\star(\unit_2)} \|T^\star(\mu)\|_H^\star.
\]
By the characterization of $B_H^\star(\unit_2)$ in~\eqref{a-BHstarunit}
and the characterization of the norm $\|\cdot\|_H^\star$ in Lemma~\ref{l-norm-rel}, we get
$$
\sup_{\mu\in B_H^\star(\unit_2)} \|T^\star(\mu)\|_H^\star
=\sup_{\nu,\pi\in \pP(\unit_2)} \|T^\star(\nu)-T^\star(\pi)\|_H^\star=\frac 1 2 \sup_{\nu,\pi\in \pP(\unit_2)} \|T^\star(\nu)-T^\star(\pi)\|_T^\star
$$
For the second equality, note that
$$
\begin{array}{ll}
\othernorm{T^\star(\nu)-T^\star(\pi)}_T^\star
&=\displaystyle\sup_{x\in[0,\unit_1]} \<T^\star(\nu)-T^\star(\pi),2x-\unit_1>\\
&=2\displaystyle\sup_{ x\in[0,\unit_1]} \<T^\star(\nu)-T^\star(\pi),x>
\end{array}.
$$
We next show that the supremum can be restricted to the set of extreme points. By the Banach-Alaoglu theorem, $B_H^\star$ is weak-star compact,
and it is obviously convex. The dual space $\cM$ endowed with the weak-star topology is a locally convex topological space. Thus by the Krein-Milman theorem, the unit ball $B_H^\star$, which is a compact convex set in 
$\cM$ with respect to the weak-star topology, is the closed convex hull of its extreme points. So every element $\rho$ of $B_H^\star(\unit_2)$
is the limit of a net  $(\rho_\alpha)_\alpha$ of elements of $\operatorname{conv}\operatorname{extr} B_H^\star(\unit_2)$, 
Observe now that the function
\[ \varphi: \mu \mapsto \|T^\star(\mu)\|_H^\star = \sup_{x\in B_H(\unit_1)} \<T^\star(\mu),x>
=\sup_{x\in B_H(\unit_1)} \<\mu, T(x)>
\]
which is a sup of weak-star continuous maps is convex and  weak-star lower
semi-continuous. This implies that $\varphi(\rho)
\leq  \liminf_\alpha \varphi(\rho_\alpha) \leq \sup_{\operatorname{conv}\operatorname{extr} B_H^\star(\unit_2)}\varphi(\mu)  = \sup_{\operatorname{extr} B_H^\star(\unit_2)}\varphi(\mu)$. 
Using the characterization of the extreme points in Proposition~\ref{p-carac-ext}, we get:
\begin{align*}
\sup_{\mu\in B_H^\star(\unit_2)} \|T^\star(\mu)\|_H^\star &= 
\sup_{\mu\in\operatorname{extr} B_H^\star(\unit_2)} \|T^\star(\mu)\|_H^\star 
=\sup_{\substack{\nu,\pi \in \operatorname{extr}\cP(\unit_2)\\ \nu\perp\pi}} \|T^\star(\nu)-T^\star(\pi)\|_H^ \star.\qedhere
\end{align*}
\end{proof}

\begin{rem}\label{rk-reached}
When $\cX_1$ is of finite dimension, the set $[0,\unit_1]$ is the convex hull of the set of its extreme points, hence, the supremum over the variable $x\in[0,\unit_1]$ in~\eqref{e-carac-TH} is attained at an extreme point. Similarly,
if $\cX_2$ is of finite dimension, the suprema over $(\nu,\pi)$ in the same
equation are also attained, because the map $\varphi$ in the proof of the previous theorem, which is a supremum of an equi-Lipschitz family of maps,
is continuous (in fact, Lipschitz).
\end{rem}
\begin{rem}\label{rem-comReeb}
 Theorem~\ref{th-opnorm} should be compared with Proposition~12 of~\cite{ReebWolf2011} which can be stated as follows.
\begin{prop}[Proposition 12 in~\cite{ReebWolf2011}]\label{p-Reeb12}
Let $L:\mathcal{V}\rightarrow \mathcal{V'}$ be a linear map and let $\mathcal{B}\subset \mathcal{V}$ and $\mathcal{B'}\subset \mathcal{V}'$ be bases. Then
\begin{align}\label{a-Reeb12}
\sup_{v_1\neq v_2\in\mathcal{B}} \frac{\othernorm{L(v_1)-L(v_2)}_{\mathcal{B'}}}{\othernorm{v_1-v_2}_{\mathcal{B}}}=\frac{1}{2} 
\sup_{v_1,v_2\in \extr \mathcal{B}}\othernorm{L(v_1)-L(v_2)}_{\mathcal{B}'}
\end{align}
\end{prop}
The first term in~\eqref{a-Reeb12} is called the \firstdef{contraction ratio} of the linear map $L$, with respect to base norms. One important applications of this proposition concerns the \firstdef{base preserving} maps $L$ such that $L(\mathcal{B})\subset \mathcal{B'}$.
Let us translate this proposition in the present setting. Consider a linear map $T:\cX/\R\unit_1\rightarrow \cX/\R\unit_2$. Then $T^\star(\pP(\unit_2))\subset \pP(\unit_2)$ is a base preserving linear map and so,
Proposition~12 of~\cite{ReebWolf2011} shows that:
\begin{align}
\sup_{\substack{\nu,\pi\in \pP(\unit_2)\\ \nu\neq \pi}} \frac{\othernorm{T^\star(\nu-\pi)}_T^\star}{\othernorm{\nu-\pi}_{T}^\star}=\frac{1}{2} \sup_{\nu,\pi
\in \extr \pP(\unit_2)} \othernorm{T^\star(\nu)-T^\star(\pi)}_{T}^\star
\label{e-disjoint}
\end{align}
Hence, by comparison with~\cite{ReebWolf2011}, the additional
information here is the equality between 
the contraction ratio in Hopf's oscillation seminorm of a unit preserving linear map, and the contraction ratio with respect to the base norms
of the dual base preserving map. The latter is the primary
object of interest in quantum information theory 
whereas the former
is of interest in the control/consensus literature.
We also proved that the supremum in~\eqref{e-disjoint} can
be restricted to pairs of {\em disjoint} extreme points $\nu,\pi$.
Finally, the expression of the contraction rate as the last supremum 
in Theorem~\ref{th-opnorm} leads here to an abstract version of Dobrushin's ergodic coefficient, see Eqn~\eqref{a-tauA} and Corollary~\ref{coro-noncommDobr} below.
\end{rem}

Recall that {\em Hilbert's projective metric} between two elements $x,y\in \C^ 0$ is defined as:
$$d_H(x,y)=\log(M(x/y)/m(x/y)).$$

Consider a linear operator $T:\cX_1\rightarrow \cX_2$ such that $T(\C_1^0)\subset \C_2^0$.
Following~\cite{birkhoff57,Bushell73}, we define the projective diameter of $T$ as below:
$$
\Diam T=\sup \{d_H(T(x),T(y)): x,y\in \C_1^0\}.
$$
The Birkhoff's contraction formula~\cite{birkhoff57,Bushell73} states that:
\begin{theo}[\cite{birkhoff57,Bushell73}]
$$
\sup_{x,y\in \C_1^0} \frac{\omega(T(x),T(y))}{\omega(x,y)}=
\sup_{x,y\in \C_1^0} \frac{d_H(T(x),T(y))}{d_H(x,y)}= \tanh(\frac{\Diam T}{4}).$$
\end{theo}
Following~\cite{ReebWolf2011}, we define the projective diameter of $T^\star$ by:
$$
\Diam T^\star=\sup \{d_H(T^\star(u),T^\star(v)):  u,v\in \C_2^\star\backslash 0\}.
$$
Note that $\Diam T=\Diam T^\star$. This is because
$$
\begin{array}{ll}
\displaystyle\sup_{x,y\in\C_1^0}  \frac{M(T(x)/T(y))}{m(T(x)/T(y))}
&=\displaystyle\sup_{x,y\in\C_1^0} \sup_{u,v\in \C_2^\star\backslash 0} \frac{\<u,T(x)>\<v,T(y)>}{\<u,T(y)>\<v,T(x)>}\\
&=
\displaystyle\sup_{u,v\in \C_2^\star\backslash 0 } \frac{M(T^ \star(u)/T^ \star(v))}{m(T^ \star(u)/T^ \star(v))}
\end{array}
$$
\begin{coro}[Compare with~\cite{ReebWolf2011}]\label{th-TdiamT}
Let $T:\cX_1 \to \cX_2$ 
be a bounded linear map such that $T(\unit_1)\in \mathbb{R} \unit_2$ and $T(\C_1^0)\subset \C_2^0$ , then:
$$
\othernorm{T^ \star}_H^ \star=\othernorm{T}_H\leq \tanh(\frac{\Diam T}{4})=\tanh(\frac{\Diam T^ \star}{4})
$$
\end{coro}
\begin{proof}
 It is sufficient to prove the inequality. For this, note that
$$
\othernorm{T}_H=\sup_{x\in \cX_1/\R\unit_1} \omega(T(x),\unit_2)/\omega(x,\unit_1)=\sup_{x \in \C_1^0} \omega(T(x),\unit_2)/\omega(x,\unit_1).
$$
Then we apply Birkhoff's contraction formula.
\end{proof}
\begin{rem}
Reeb et al~\cite{ReebWolf2011} showed in a different way that
$$
\othernorm{T^\star}_H^\star\leq \tanh(\frac{\Diam T^ \star}{4}).
$$
The proof above shows that as soon as the duality formula
$\othernorm{T^ \star}_H^ \star=\othernorm{T}_H$ has been
obtained, the latter inequality follows from Birkhoff
contraction formula.
\end{rem}
Nussbaum~\cite{nussbaum94} 
showed that the Lipschitz constant  in Hilbert's projective metric of a non-linear map is determined by the operator norm of its derivative
with respect to Hopf's oscillation seminorm. We first use this
result to deduce a characterization of the contraction rate of non linear maps in Hilbert's metric. We first quote the result of~\cite{nussbaum94} which we shall use.
\begin{theo}[Coro 2.1,~\cite{nussbaum94}]\label{th-Nus94}
 Let $U\subset \C^ 0$ be a convex open set such that $tU\subset U$ for all $t>0$. Let $f:U\rightarrow \C^ 0$ be a continuously differentiable map
such that $\omega(f(x)/f(y))=0$ whenever $x,y\in U$ and $\omega(x/y)=0$. For each $x\in U$ define $\lambda(x)$, $\lambda_0$ and $k_0$ by:
$$
\lambda(x):=\inf\{c>0: \myhilbert(Df(x)v/f(x))\leq c \myhilbert(v/x) \mathrm{~for~all~}v\in \cX\},
$$
$$
\lambda_0:=\sup\{\lambda(x):x\in U\},
$$
$$
k_0:=\inf\{c>0: d_H(f(x),f(y))\leq c d_H(x,y)\mathrm{~for~all~}x,y\in U\}.
$$
Then it follows that $\lambda_0=k_0$.
\end{theo}
Then a direct corollary of Theorem~\ref{th-opnorm} and~\ref{th-Nus94} yields the Lipschitz constant in Hilbert's metric of a (non-linear) map.
\begin{coro}\label{coro-Hilbertmetricnonlinearmap}
 Let $U$ and $f$ be as in Theorem~\ref{th-Nus94}. Then:
$$
\sup_{\substack{x,y\in U\\d_H(x,y)\neq 0}}
 \frac{d_H(f(x),f(y))}{d_H(x,y)}=\sup_{x\in U} \sup_{\nu,\pi\in \extr\pP(f(x))}\sup_{z\in[0,x]} \<\nu-\pi,Df(x)z>.
$$
\end{coro}
\begin{rem}
This corollary generalizes Corollary~2.1 of~\cite{nussbaum94}, which gives
a similar characterization in terms of extreme points, when $\cX=\mathbb{R}^n$
and $\C=\mathbb{R}_+^n$.
Note that 
in the finite dimensional case, the suprema over the variable $z$ 
and over the variables $\nu,\pi$ are attained
(see Remark~\ref{rk-reached}). Moreover, the supremum over $z$ 
is attained at an extreme point of $[0,x]$.
\end{rem}

\section{Application to discrete consensus operators on cones}\label{sec-consensusoperator}
A classical result, which goes back to D\oe blin and Dobrushin, characterizes
the Lipschitz constant of a Markov matrix acting on the space
of measures (i.e., a row stochastic matrix acting on the left), 
with respect to the total variation norm (see the discussion
in Remark~\ref{rem6} below). The same constant characterizes the contraction ratio with respect to the ``diameter'' (Hopf oscillation seminorm)
of the consensus system driven by this Markov matrix (i.e., a row
stochastic matrix acting on the right).
Consensus operators on cones extend Markov matrices.
In this section, we extend to these abstract operators
a number of known properties of Markov matrices.

A linear  map $T:\cX\to \cX$ is a {\em consensus operator} with respect to
a unit vector $\unit$ in the interior $\C^ 0$ of a closed convex pointed cone $\C\subset \cX$ if it satisfies the two following properties: 
\begin{itemize}
\item[(i)] $T$ is positive, i.e., $T(\C) \subset \C$.
\item[(ii)] $T$ preserves the unit element $\unit$, i.e., $T(\unit)=\unit$.
\end{itemize}

\begin{example}\label{ex-markovmt}
When $\cX=\R^n$, $\C$ is the standard orthant and $\unit$ is the standard unit vector $\bold 1$ (Example~\ref{rem1}), a linear map $T(x)=Ax$ is a consensus operator if and only if $A$ is a row stochastic matrix. The operator norm is the contraction rate of the matrix $A$ with respect to the diameter $\Delta$:
$$
\othernorm{T}_H=\tau(A):=\sup_x \frac{\Delta(Ax)}{\Delta(x)},
$$
and the dual operator norm is the Lipschitz constant of $A^{\top}$ on $\pP(\bold 1)$ with respect to the total variation distance:
$$
\othernorm{T}_H^\star=\delta(A):=\sup_{\mu,\nu \in \pP(\bold 1)} \frac{\othernorm{A^{\top}\mu -A^{\top}\nu}_{TV}}{\othernorm{\mu-\nu}_{TV}}
$$
The value $\tau(A)$ allows one to bound the convergence rate of the stationary linear consensus system the dynamics of which is given by the matrix $A$, \cite{SpeilmanMorse,Blondel05convergencein}.
 The value $\delta (A)$ is known 
as the {\em ergodicity coefficient} of the Markov chain with transition probability matrix $A^{\top}$, see~\cite{PeresLevin}.
\end{example}
\begin{example}\label{ex-krausmap}
 When $\cX=\sym_n$, $\C=\sym_n^+$ and $\unit= I_n$ (Example~\ref{ex-Symn}), the linear map $\Phi:\sym_n\rightarrow \sym_n$ defined by
\begin{align}\label{a-PhiX}
\Phi(X)=\sum_{i=1}^m V_i^*XV_i,\qquad \sum_{i=1}^m V_i^*V_i=I_n
\end{align}
is a consensus operator. The dual operator is then given by:
$$
\Psi(X)=\sum_{i=1}^* V_i XV_i^*.
$$
Both maps are completely positive. They represent a purely quantum channel~\cite{ReebWolf2011,sepulchre}.
The map $\Phi$ is unital and acts between spaces of operators while the adjoint map $\Psi$ is trace-preserving and acts 
between spaces of states (density matrices). The operator norm of $\Phi: \sym_n/\R I_n \rightarrow \sym_n/\R I_n$ is the contraction rate of the diameter of the spectrum:
$$
\othernorm{\Phi}_H=\sup_{X\in \sym_n} \frac{\lambda_{\max}(\Phi(X))-\lambda_{\min}(\Phi(X))}{\lambda_{\max}(X)-\lambda_{\min}(X)}.
$$
The operator norm of the adjoint map $\Psi: \pP(I_n)\rightarrow \pP(I_n)$ is the contraction rate of the trace distance:
$$
\othernorm{\Psi}_H^\star=\sup_{\rho_1,\rho_2\in \pP(I_n)}\frac{\othernorm{\Psi(\rho_1)-\Psi(\rho_2)}_1}{\othernorm{\rho_1-\rho_2}_1} .
$$
The value $\othernorm{\Phi}_H$ and $\othernorm{\Psi}_H^\star$ are the noncommutative counterparts of
$\tau(\cdot)$ and $\delta(\cdot)$.
\end{example}

 A direct application of Theorem~\ref{th-opnorm} leads to following
characterization of operator norm, which will be seen
to extend Dobrushin's formula (see Remark~\ref{rem6} below).
\begin{coro}\label{coro-markTH}
 Let $T: \cX \to \cX$ be a consensus operator with respect to $\unit$. Then,
\[
\|T\|_H = \|T^\star\|_H^\star =1-\inf_{\substack{\nu,\pi \in \operatorname{extr}\cP(\unit)\\ \nu\perp\pi}}\inf_{x\in[0, \unit]} \<\pi,T(x)>+\<\nu, T(\unit-x)>.
\]
\end{coro}
\begin{proof}
Since $T(\unit)=\unit$, we have: 
$$
\sup_{\substack{\nu,\pi \in \operatorname{extr}\cP(\unit)\\ \nu\perp\pi}} \sup_{ x\in[0, \unit]}
\<\nu-\pi,T(x)>=\sup_{\substack{\nu,\pi \in \operatorname{extr}\cP(\unit)\\ \nu\perp\pi}}\sup_{x\in[0, \unit]} 1-\<\pi,T(x)>-\<\nu, T(\unit-x)>.
$$
\end{proof}

\begin{rem}
In the finite dimensional case, as already noted in Remark~\ref{rk-reached}, the supremum is reached at $\extr[0,\unit]$.
\end{rem}

\begin{rem}\label{rem6}
 In the case of a stochastic matrix $A$ (Example~\ref{ex-markovmt}), Corollary~\ref{coro-markTH} implies that:
\begin{align}\label{a-rnn}
\tau(A)=\delta(A)=\frac{1}{2}\sup_{i\neq j} \|A^{\top}e_i- A^{\top}e_j\|_1.
\end{align}
This is a known result in the study of Markov chain~\cite{Seneta90}. 
The value $\tau(A)$ is known under the name of {\em Dobrushin's ergodic coefficient} of the stochastic matrix $A$~\cite{Dobrushin56}. It is explicitly given by:
\begin{align}
\tau(A)&=
1-\displaystyle\min_{i\neq j}\sum_{s=1}^n \min(A_{is},A_{js}).
\label{a-tauA}
\end{align}
Indeed, the characterization of $\tau(A)=\|T\|_H$ by the last supremum
in Corollary~\ref{coro-markTH} yields
\begin{align*}
\tau(A)&=1-\min_{i\neq j} \min_{I\subset\{1,\dots,n\}}(\sum_{k\in I} A_{ik}+\sum_{k\notin I}A_{jk})\\&
\end{align*}
from which~\eqref{a-tauA} follows. 

A simple classical situation in which $\tau(A)<1$ is when
there is a {\em D\oe blin state}, i.e.,
an element $j\in\{1,\dots,n\}$ 
such that $A_{ij}>0$ holds for all $i\in\{1,\dots,n\}$. 
\end{rem}

Specializing Corollary~\ref{coro-markTH} to the case of quantum channels (Example~\ref{ex-krausmap}), we 
obtain the noncommutative version of Dobrushin's ergodic coefficient. 
\begin{coro}\label{coro-noncommDobr}
Let $\Phi$ be a quantum channel defined in~\eqref{a-PhiX}. 
Then,
 \begin{align}\label{a-PhiHPsiH}
\othernorm{\Phi}_H=\othernorm{\Psi}_H^\star=1-
\displaystyle\min_{\substack{u,v:u^*v=0\\u^*u=v^*v=1}}\min_{\substack{X=(x_1,\dots,x_n)\\XX^*=I_n}} \sum_{i=1}^ n \min \{u^*\Phi(x_ix_i^ *)u,v^*\Phi(x_ix_i^ *)v\}
\end{align}
\end{coro}
\begin{proof}
It can be easily checked that 
$$
\extr [0,I_n]=\{P\in \sym_n:P^ 2=P\}.
$$
Hence, Corollary~\ref{coro-markTH} and Remark~\ref{rem-xx*yy*}
yield:
\begin{align*}
\othernorm{\Phi}_H=\othernorm{\Psi}_H^\star&=
1-\displaystyle\min_{\substack{u,v:u^*v=0\\u^*u=v^*v=1}}\min_{\substack{Y^ 2=Y}}u^*\Phi(I_n-Y)u+v^*\Phi(Y)v\\
&=1-\displaystyle\min_{\substack{u,v:u^*v=0\\u^*u=v^*v=1}}\min_{\substack{X=(x_1,\dots,x_n)\\XX^*=I_n}} \min_{J\subset\{1,\dots,n\}} \sum_{i\in J} u^*\Phi(x_ix_i^ *)u+ \sum_{i\notin J}v^*\Phi(x_ix_i^ *)v \\
\end{align*}
from which~\eqref{a-PhiHPsiH} follows.\qedhere
\end{proof}

We now make the following basic observations for a consensus operator $T:\cX\rightarrow \cX$:
$$
M(T(x)/\unit)\leq M(x/\unit),\enspace m(T(x)/\unit)\geq m(x/\unit),\enspace \forall x\in \cX.
$$
It follows that $\othernorm{T}_H\leq 1$.
The case when $\othernorm{T}_H<1$ or equivalently $\othernorm{T^\star}_{H}^\star<1$
is of special interest, as shown by the following theorem, which
shows that the iterates of $T$ convergence to a rank one projector
with a rate bounded by $\|T\|_H$.
\begin{theo}[Geometric convergence to consensus]\label{th-ex-con}
If $\othernorm{T}_H<1$ or equivalently $\othernorm{T^\star}_{H}^\star<1$, then there is $\pi \in \pP(\unit)$ 
such that for all $x\in \cX$
\[
\|T^n(x)-\<\pi,x>\unit\|_T \leq (\othernorm{T}_H)^n \|x\|_H, 
\]
and for all $\mu \in \pP(\unit)$
\[
\|(T^\star)^n(\mu)-\pi \|_H^\star \leq (\othernorm{T}_H)^n.
\]
\end{theo}
\begin{proof}
The intersection $$\displaystyle\cap_{n} [m(T^n(x)/\unit),M(T^n(x)/\unit)]\subset \R$$ is
nonempty (as a non-increasing intersection of nonempty compact sets), and since $\othernorm{T}_H<1$ and
$$\omega(T^n(x)/\unit)\leq (\othernorm{T}_H)^n \omega(x/\unit),$$ this intersection must
be reduced to a real $\{c(x)\}\subset \R$ depending on $x$, i.e.,
$$c(x)= \displaystyle\cap_{n} [m(T^n(x)/\unit),M(T^n(x)/\unit)].$$
Thus for all $n\in \mathbb N$,
$$
-\omega(T^n(x)/\unit) \unit\leq T^n(x)-c(x)\unit\leq \omega(T^n(x)/\unit)\unit.
$$
Therefore by definition:
$$
\|T^n(x)-c(x)\unit \|_T\leq \omega(T^n(x)/\unit).
$$
Then we get:
$$
\|T^n(x)-c(x)\unit\|_T \leq (\othernorm{T}_H)^n \othernorm{x}_H.
$$
It is immediate that:
$$
c(x)\unit=\lim_{n\rightarrow \infty} T^n(x)
$$
from which we deduce that $c:\cX\rightarrow \R$ is a continuous linear functional. Thus there is $\pi\in \cX^\star$ such that $c(x)=\<\pi,x>$.  Besides it is immediate that $\<\pi,\unit>=1$ and $\pi\in \C^\star$ because
$$
x\in \C \Rightarrow c(x)\unit \in \C\Rightarrow  c(x)\geq 0\Rightarrow \<\pi,x>\geq 0.
$$
Therefore $\pi\in \pP$.
Finally for all $\mu \in \pP$ and all $x\in \cX$ we have
$$
\begin{array}{ll}
\<(T^\star)^n(\mu)-\pi,x>&=\<\mu, T^n(x)-\<\pi,x>\unit>\\
&\leq \othernorm{\mu}_T^\star\othernorm{T^n(x)-\<\pi,x>\unit}_{T}\\
&\leq (\othernorm{T}_H)^n \|x\|_H.
\end{array}
$$
Hence
$$
\othernorm{(T^\star)^n(\mu)-\pi}_{H}^\star \leq (\othernorm{T}_H)^n.
$$
\end{proof}
\begin{rem}
Specializing Theorem~\ref{th-ex-con} to the case of $\cX=\R^n$ (Example~\ref{rem1}) we obtain that if $\tau(A)=\delta(A)<1$, then 
 $$A^n \rightarrow \bold 1\pi^T,\enspace n\rightarrow +\infty$$ where $\pi$
is the unique invariant measure of the stochastic matrix $A$. This is a well-known result in the study of ergodicity property and mixing times of Markov chains, see for example~\cite{Seneta90} and~\cite{PeresLevin}.
\end{rem}

\begin{rem}
 A time-dependent consensus system is described by
\begin{align}\label{a-tmedeconsesys}
x_{k+1}=T_{k+1}(x_k),\quad k\in \mathbb N
\end{align}
where $\{T_k:k\geq 1\}$ is a sequence of consensus operators sharing a common unit element $\unit\in \C^ 0$. Then if there is an integer $p>0$ and
a constant $\alpha<1$ such that for all 
$i\in \mathbb N$
$$
\othernorm{T_{i+p}\dots T_{i+1}}_H\leq \alpha,
$$
then the same lines of proof of Theorem~\ref{th-ex-con} imply the existence of $\pi \in \pP(\unit)$ such that for all $\{x_k\}$ satisfying~\eqref{a-tmedeconsesys},
$$
\othernorm{x_{k} -\<\pi, x_{0}>\unit}_T\leq \alpha^{\lfloor{\frac{k}{p}}\rfloor} \othernorm{x_{0}}_H,
\quad n\in \mathbb N.
$$
\end{rem}
\begin{rem}\label{rem-tmedepenran}
In the case of $\cX=\R^ n$ and  $T_k(x)=A_k x$ where $A_k$ is a stochastic matrix, Moreau~\cite{Moreau05} showed that 
if all the non-zero entries are bounded from below by a positive constant 
and if there is $p\in \mathbb N$ such that for all $i\in \mathbb N$ there is a node connected to all other nodes in the graph associated to the matrix $A_{i+p}\dots A_{i+1}$, then the system~\ref{a-tmedeconsesys} is globally uniformly convergent. These two conditions imply exactly that the Dobrushin's ergodic coefficient~\eqref{a-tauA} of $A_{i+p}\dots A_{i+1}$, which is also the operator norm $\othernorm{T_{i+p}\dots T_{i+1}}_H$, is bounded by a constant less than 1.
\end{rem}

\section{The contraction rate in Hopf's oscillation of a linear flow}\label{sec-linearequation}
\subsection{Abstract formula for the contraction rate}
Hereinafter, we only consider a {\em finite dimensional} vector space $\cX$.
The set of linear transformations on $\cX$ is denoted by $\End(\cX)$.
 Let $L\in \End(\cX)$ such that $L(\unit)=0$. 
The next proposition characterizes the contraction rate of the flow
associated
to the 
linear differential equation
$$
\dot x=L(x),
$$
with respect to Hopf's oscillation seminorm.
\begin{prop}\label{pr-alphaHilsemiL}
 The optimal constant $\alpha$ such that
$$
\othernorm{\exp(tL)x}_H\leq e^{\alpha t}\othernorm{x}_H,\enspace \forall t\geq 0, x\in \cX
$$
 is 
\begin{align}\label{a-hL}
h(L):=-\inf_{\nu,\pi\in \extr\pP(\unit)}\inf_{\substack{x \in \extr([0, \unit]) \\ \<\nu,x>+\<\pi,\unit-x>=0}} \<\nu,L(x)>+\<\pi,L(\unit-x)>.
\end{align}
\end{prop}
\begin{proof}
Let $I:\cX\rightarrow \cX$ denote the identity
transformation.
We define a functional on $\End(\cX)$ by:
$$
F(W)=\sup_{\nu,\pi\in \pP(\unit)}\sup_{x\in [0,\unit]}\<\pi-\nu,W(x)>
$$
By Theorem~\ref{th-opnorm}, the optimal constant $\alpha $ is:
\begin{align}
\alpha &=\displaystyle\lim_{\epsilon \rightarrow 0^+} \epsilon^{-1}(\othernorm{\exp(\epsilon L)}_H-1)\nonumber\\
 &=\displaystyle\lim_{\epsilon \rightarrow 0^+} \epsilon^{-1}(F(\exp(\epsilon L))-F(I))
\enspace .\label{e-semi}
\end{align}
Recall that a map is said to be {\em semidifferentiable} at a point if it has one-sided directional derivatives in all directions, and if
the limit defining the one-sided directional derivative is uniform
in the direction, see Definition 7.20 of~\cite{RockafellarWets}, to
which we refer for information on the different notions used here. 
The limit in~\eqref{e-semi} coincides with to the semiderivative of $F$ at point $I$  in the direction $L$ if $F$ is semidifferentiable. We next show
that it is so, and compute the limit.
Since we assume that $\pP(\unit)$ and $[0,\unit]$ are compact sets and the function
$$
F_{\nu,\pi,x}(W)=\<\pi-\nu,W(x)>
$$
is continuously differentiable on $W$ such that $F_{\nu,\pi,x}(W)$ and
$
DF_{\nu,\pi,x}(W)
$ are jointly continuous on $(\nu,\pi,x,W)$, we know that
 $F:\End(\cX)\rightarrow \R$ defines a subsmooth function (see~\cite[Def 10.29]{RockafellarWets} therefore $F$ is semidifferentiable and the semiderivative of $F$ at point $I$ in the direction $L$ equals to  (see~\cite[Thm 10.30]{RockafellarWets}) 
$$
\begin{array}{ll}
DF(I)(L)=\displaystyle\sup_{\nu,\pi,x\in T(I)} \<\pi-\nu,L(x)>\\
\end{array}
$$
where $$T(I)=\argmax{x\in[0,\unit],\nu,\pi\in\pP(\unit)} F_{\nu,\pi,x}(I).$$
 Hence,
$$
\begin{array}{ll}
\alpha &=DF(I)(L)\\
&=\displaystyle\sup_{\nu,\pi\in \pP(\unit)}\sup_{\substack{ x\in [0, \unit]\\ \<\pi-\nu,x>=1}} \<\pi-\nu,L(x)>\\
&=-\displaystyle\inf_{\nu,\pi\in \pP(\unit)}\inf_{\substack{ x\in [0, \unit]\\ \<\nu,x>+\<\pi,\unit-x>=0}} \<\nu,L(x)>+\<\pi,L(\unit-x)>.
\end{array}
$$
Since $\cX$ is finite dimensional, the sets $\pP(\unit)$ and $[0,\unit]$ are
both compact, and they are the convex hull of their extreme points. 
Henceforth, arguing
as in Remark~\ref{rk-reached} above, we
can replace $\pP(\unit)$ and $[0, \unit]$ by $\extr \pP(\unit)$ and $ \extr([0,\unit])$, respectively.
\end{proof}

\subsection{Contraction rate in $\mathbb{R}^ n$}\label{subsec-formulainRn}
One may specialize Formula~\eqref{a-hL} to the case $\cX=\R^ n$, $\C=\R^ n_+$ and $\unit=\bold 1$.
For $x\in \R^ n$ we denote by $\delta(x)$ the diagonal
matrix with entries $x$. 
\begin{coro}
  Let $A$ be a square matrix such that $A\bold 1=0$.
Then
\begin{align}\label{a-hA}
h(A)=-\min_{i\neq j} \big(A_{ji}+A_{ij}+\sum_{k\notin\{i,j\}} \min(A_{ik},A_{jk})\big).
\qquad\qquad\qquad\qquad\qquad\qquad\;\qed
\end{align}
\end{coro}
\begin{proof}
 Recall that
$$
\extr(\pP(\bold 1))=\{e_i:i=1,\dots,n\},\enspace \extr[0,\bold 1]=\{\sum_{i\in I}e_i:I\subset \{1,\dots,n\}\}.
$$
Therefore we have:
$$
\begin{array}{ll}
h(A)&=-\displaystyle\min_{i\neq j} \min_{\substack{I\subset \{1,\dots,n\}\\ i\notin I,j\in I}} \sum_{k\in I} A_{ik}+\sum_{k\notin I} A_{jk}
\\&=-\displaystyle\min_{i\neq j} A_{ij}+A_{ji}+ \min_{\substack{I\subset \{1,\dots,n\}\\ i\notin I,j\in I}} \sum_{k\in I\backslash \{j\}} A_{ik}+\sum_{k\notin I\cup\{i\}} A_{jk}
\\&=-\displaystyle \min_{i\neq j} A_{ij}+A_{ji}+\sum_{k\notin\{i,j\}} \min(A_{ik},A_{jk}).
\end{array}
$$
\end{proof}

\begin{rem}\label{rem-compareMoreaulinear}
Consider the order-preserving case, i.e.\ $A_{ij}\geq 0$ for $i\neq j$. 
Such situation was studied extensively in the context of consensus
dynamics. In particular, let $G=(V,E)$ be a graph and
equip each arc $(i,j)\in E$ a weight $C_{ij}>0$ (the node $j$ is connected to $i$). One of the consensus systems that Moreau~\cite{Moreau05} studied is:
$$
\dot x_i= \sum_{(i,j)\in E} C_{ij}(x_j-x_i),\enspace i=1,\dots,n \enspace.
$$
This can be written as $\dot x = Ax$, where $A_{ij}=C_{ij}$
for $i\neq j$ and $A_{ii}=\sum_{j}C_{ij}$ is a {\em discrete Laplacian}.
A general result of Moreau implies that if there is a node connected by path to all other nodes in the graph $G$, then the system is globally convergent.
Our results show that if $h(C)<0$ then the system converges exponentially to consensus with rate $h(C)$.  The condition $h(C)=0$ means that 
there are two nodes disconnected with each other ($C_{ij}+C_{ji}=0$) and all other nodes are connected by arc to at most one of them ($\sum_{k\notin\{i,j\}} \min(C_{ik},C_{jk})=0$). 
The condition $h(C)<0$, though more strict
than Moreau's connectivity condition, gives an explicit contraction rate.
\end{rem}
\begin{rem}
 In addition, our result applies to not necessarily order-preserving flows. For example, consider the matrix
$$A=\left(\begin{array}{lll} -3&1&2\\1&0&-1\\1&1&-2\end{array}\right).$$
A basic calculus shows that $h(A)=-1$. Therefore, every orbit of the linear system $\dot x=Ax$ converges exponentially with rate $-1$ to a multiple of the unit vector. 
\end{rem}

\begin{rem}\label{rem-hdefined}
We point out that as a contraction constant, $h(A)$ makes sense only when
$A\bold 1=0$. However, as a functional $h$ is well defined on the 
space of square matrices. Moreover, since the diagonal elements do not account in the
formula~\eqref{a-hA}, it is clear that for any square matrix $B\in \mM_n(\R)$ and $x\in \R^ n$
$$
h(B)=h(B-\delta (x)).
$$
\end{rem}
\subsection{Contraction in the space of Hermitian matrices}
We now specialize Formula~\eqref{a-hL} to the case $\cX=\sym_n$, $\C=\sym_n^ +$ and $\unit=I_n$:
\begin{coro}
 Let $\Phi:\sym_n\rightarrow \sym_n$ 
be a linear application such that $\Phi(I_n)=0$.
Then
\begin{align}\label{a-hPhi}
h(\Phi)=-\inf_{\substack{X=(x_1,\dots,x_n)\\XX^*=I_n}} \big( x_1^*\Phi(x_2x_2^ *)x_1 +x_2^*\Phi(x_1x_1^ *)x_2+\sum_{k=3}^ n \min(x_1^*\Phi(x_kx_k^ *)x_1 , x_2^*\Phi(x_kx_k^ *)x_2)\big).
\end{align}
where $x_i$ is the $i$-th column vector of each unitary matrix $X$.\qed
\end{coro}
\begin{proof}
 Recall that 
$$
\extr(\pP(I_n))=\{xx^ *: x\in \CC^ n, x^ *x=1\},\enspace \extr[0, I_n]=\{P\in \sym_n:\, P^ 2=P\}.
$$
Then,
\begin{align*}
h(\Phi)&=-\displaystyle\inf_{\substack{x_1^ *x_1=x_2^ *x_2=1}} \inf_{\substack{P^ 2=P\\Px_1=0,Px_2=x_2}}  x_1^*\Phi(P)x_1+x_2^ * \Phi(I_n-P)x_2
\\&=-\displaystyle\inf_{\substack{x_1^ *x_1=x_2^ *x_2=1}} \inf_{\substack{P=(x_2,x_3,\dots,x_k)\\P^ 2=P,Px_1=0}} \sum_{i=2}^ k x_1^* \Phi(x_ix_i^ *)x_1+ x_2^ * \Phi(I_n-P)x_2
\\&=-\displaystyle\big(\inf_{\substack{x_1^ *x_1=x_2^ *x_2=1}} x_1^ *\Phi(x_2x_2^ *)x_1+x_2^ *\Phi(x_1x_1^ *)x_2\\&~~~~\qquad+
\displaystyle\inf_{\substack{X=(x_1,x_2,\dots,x_n)\\XX^*=I_n}} 
\sum_{i=3}^ k x_1^*\Phi(x_ix_i^ *)x_1+ \sum_{i=k+1}^ n x_2^ * \Phi(x_ix_i^ *)x_2\big). \qedhere
\end{align*}
\end{proof}

As pointed out in Remark~\ref{rem-hdefined}, $h$ is 
a functional well defined for all linear applications from $\sym_n$ to $\sym_n$. It is
interesting to remark that for any linear application $\Psi$ and any square matrix $Z$, 
$$
h(\Psi)=h(\Phi)
$$
where $\Phi(X)=\Psi(X)-ZX-XZ$ for all $X\in \sym_n$.
\subsection{Contraction rate of time-dependent linear flows}
We now state the result analogous to Proposition~\ref{pr-alphaHilsemiL}, which applies to \firstdef{time dependent} linear flows.
Let $t_0>0$ and $L_{\cdot}(\cdot): [0,t_0)\times \cX\rightarrow \cX$ be a continuous application linear in the second variable
such that $L_t(\unit)=0$ for all $t\in[0,t_0)$. We denote by $U(s,t)$  the evolution operator of the following linear time-varying differential equation:
$$
\dot x(t)=L_t(x),\enspace t\in[0,t_0).
$$
Then a slight modification of the proof of Proposition~\ref{pr-alphaHilsemiL} leads to the following result.
\begin{prop}\label{pr-alphaHilsemiL2}
 The optimal constant $\alpha$ such that
$$
\othernorm{U(s,t)x}_H\leq e^{\alpha (t-s)}\othernorm{x}_H,\enspace \forall s, t\in[0,t_0), x\in \cX.
$$
 is 
\begin{align}\label{a-hL2}
\sup_{t\in[0,t_0)}h(L_t)=-\inf_{t\in[0,t_0)}\inf_{\nu,\pi\in \extr\pP(\unit)}\inf_{\substack{x \in \extr([0, \unit]) \\ \<\nu,x>+\<\pi,\unit-x>=0}} \<\nu,L_t(x)>+\<\pi,L_t(\unit-x)>.
\end{align}
\end{prop}

\section{Contraction rate in Hopf's oscillation seminorm of nonlinear flows}\label{sec-nonlinearconsensus}

Let us consider a differentiable application $\phi:\cX\rightarrow \cX$.
Since $\phi$ is locally Lipschitz, we know that for all $x_0\in \cX$, there is a maximal interval $J(x_0)$ such 
that a unique solution $x(t;x_0)$ of
\begin{align}\label{a-xtphit2}
 \dot x(t)=\phi(x(t)),\quad x(0)=x_0
\end{align}
 is defined on $J(x_0)$. We define an application $M_{\cdot}(\cdot): \R \times \cX\rightarrow \cX$ by:
$$
M_t(x_0)=x(t;x_0),\quad t\in J(x_0).
$$
The application $M$ is the flow of the equation~\eqref{a-xtphit2} and it may not be everywhere defined on $\R \times \cX$.
Since $\phi$ is continuously differentiable, the flow is differentiable with respect to the second variable. We denote by 
$DM_t(x)$ the derivative of the application $M$ with respect to the second variable at point $(t,x)$. Recall that
$$
\dot{DM_t(x)z}=D\phi(M_t(x))(DM_t(x)z),\enspace t\in J(x),z\in \cX.
$$ 
Let $U\subset \cX$ be a convex open set. For $x_0\in U$ define:
$$
t_U(x_0):=\sup\{t_0\leq J(x_0):x(t;x_0)\in U ,\enspace \forall t\in [0,t_0)\}
$$
the time when the solution of~\eqref{a-xtphit2} leaves $U$.

 Suppose that $\phi$ satisfies $\phi(x+\lambda \unit)=\phi(x)$ for all $\lambda\in\R$ and $x\in \cX$.
By uniqueness of the solution, it is clear that for all $x_0\in \cX$ and $\lambda\in \R$,
$$
M_t(x_0+\lambda \unit)= M_t(x_0)+\lambda \unit,\quad t\in J(x_0).
$$
We define the contraction rate of the flow on $U$ with respect to Hopf's oscillation seminorm:
\begin{align}\label{a-opconrate2}
\begin{array}{l}
 \alpha(U):=
\inf\{\beta\in \R: \othernorm{M_t(x)-M_t(y)}_H\leq e^{\beta t}\othernorm{x-y}_H, x,y\in U,
t\leq t_U(x)\wedge t_U(y)
\}.
\end{array}
\end{align}
\begin{theo}\label{th-conratnonlsemi}Let $\phi$ satisfy the above conditions.
Then we have
$$
\alpha(U)= \sup_{x\in U} h(D\phi(x))
$$
where $h$ is defined in~\eqref{a-hL}.
\end{theo}
\begin{proof}
Denote $$
\beta=\sup_{x\in U} h(D\phi(x)).
$$
For any $x\in U$, define
$$
L_t=D\phi(M_t(x)),\enspace t\in[0,t_U(x)).
$$
Let any $z\in \cX$.
Then $DM_t(x)z:t\in[0,t_U(x))$ is the solution of the following linear time-varying differential equation:
$$
\left\{
\begin{array}{ll}
\dot x=L_t(x), \enspace t\in[0,t_U(x)),\\
x(0)=z.
\end{array}\right.
$$
By Proposition~\ref{pr-alphaHilsemiL2}, it is immediate that for all $z\in \cX$,
$$
\omega(DM_t(x)z/\unit)\leq e^{\beta t} \omega(z/\unit),\enspace t\in[0,t_U(x)).
$$
Let $x,y\in U$ and $h<t_U(x)\wedge t_U(y)$. Denote $\gamma(s)=sx+(1-s)y:s\in[0,1]$. Then,
$$
\omega(M_h(x)-M_h(y)/\unit)\leq \int_0^1 \omega(DM_h(\gamma(s))(x-y)/\unit) ds \leq e^{\beta h} \omega(x-y/\unit). 
$$
Therefore, for all $x,y\in U$,
$$
\limsup_{h\rightarrow 0^+}\frac{\othernorm{M_h(x)-M_h(y)}_H}{h} \leq \beta \othernorm{x-y}_H.
$$
We deduce that for all $x,y\in U$ and $t<t_U(x)\wedge t_U(y)$,
$$
\limsup_{h\rightarrow 0^+} \frac{\othernorm{M_{t+h}(x)-M_{t+h}(y)}_H}{h}\leq \beta \othernorm{M_t(x)-M_t(y)}_H.
$$
Therefore, 
$$
\othernorm{M_t(x)-M_t(y)}_H\leq e^{\beta t} \othernorm{x-y}_H,\enspace t<t_U(x)\wedge t_U(y).
$$
This implies that
$$
\alpha(U)\leq \beta.
$$
Inversely, for all $x\in U$, there is $t_0>0$ such that for all $h\leq t_0$, $z\in \cX$,
$$
\othernorm{M_h(x+z)-M_h(x)}_H\leq e^{\alpha(U) h} \othernorm{z}_H.
$$
Therefore,
$$
\begin{array}{ll}
\othernorm{DM_h(x)(z)}_H&=\displaystyle\othernorm{\lim_{t\rightarrow 0^ +}\frac{M_h(x+tz)-M_h(x)}{t}}_H\\&=\displaystyle\lim_{t\rightarrow 0^ +} \frac{\othernorm{M_h(x+tz)-M_h(x)}_H}{t} \leq e^{\alpha(U)h} \othernorm{z}_H.
\end{array}
$$
By Theorem~\ref{th-opnorm}, we obtain that for $h\leq t_0$,
$$
\sup_{\nu,\pi\in \pP(\unit)}\sup_{z\in[0,\unit]} \<\nu-\pi,DM_h(x)z>\leq e^{\alpha(U) h}.
$$
It is then immediate that for $h\leq t_0$,
$$
\sup_{\nu,\pi\in \extr\pP(\unit)}\sup_{\substack{z \in \extr([0, \unit])\\ \<\nu,z>+\<\pi,\unit-z>=0}}
-\<\nu, DM_h(x)(\unit-z)>-\<\pi,DM_h(x)z>\leq e^{\alpha(U)h}-1.
$$
Dividing the two sides by $h$ and passing to the limit as $h\to 0$ we get:
$$
h(D\phi(x))\leq \alpha(U).
$$
Therefore $\beta\leq \alpha(U)$.
\end{proof}

\subsection{Applications to non-linear consensus in $\mathbb{R}^n$}\label{sec-kuramoto}
Let $G=(V,E)$ denote a directed graph. Let us equip every arc $(i,j)\in E$
with a weight $C_{ij}>0$. For $(i,j)\notin E$, we set $C_{i,j}=0$.

\begin{example}{(Non linear consensus)}\label{ex-nonlinearconsensus}
Consider the following nonlinear consensus protocol~\cite{Saber03}:
\begin{align}\label{a-nonlinearconsensus}
\dot x_k=\sum_{(i,k)\in E} C_{ik}\gamma_{ik}(x_i-x_k),\enspace k=1,\dots,n,
\end{align}
where we suppose that every map $\gamma_{ik}:\R^ n\rightarrow \R$ is differentiable. 
When every $\gamma_{ik}$ is the identity
map, the operator at the right hand-side of~\eqref{a-nonlinearconsensus} is the discrete
Laplacian of the digraph $G$, in which $C_{ik}$ is the conductivity of arc $(i,k)$.
\begin{prop}\label{pr-Uconvexopenbold1}
Let $w>0$. Suppose that
\begin{align}\label{a-gammaik}
\alpha:=\inf\{\gamma_{ik}'(t):t\in [-w,w],(i,k)\in E\}\geq 0.
\end{align} Consider the convex open set 
$$
U(w)=\{x: \othernorm{x}_H<w\}.
$$
For $x(0)\in U(w)$, the solution of~\eqref{a-nonlinearconsensus} satisfies:
$$
\othernorm{x(t)}_H\leq e^{h(C)\alpha t}\othernorm{x(0)}_H,\enspace \forall t\geq 0.
$$
\end{prop}
 \begin{proof}
For all $x\in U$,
$$
h(D\phi(x))=-\min_{i\neq j} \frac{\partial \phi_i(x)}{\partial x_j}+\frac{\partial \phi_j(x)}{\partial x_i}
+\sum_{k\neq i,j}\min(\frac{\partial \phi_i(x)}{\partial x_k},\frac{\partial \phi_j(x)}{\partial x_k}),
$$
where 
$$
\frac{\partial \phi_i(x)}{\partial x_j}=C_{ij}\gamma_{ij}'(x_j-x_i),\enspace i\neq j.
$$
Hence for all $x\in U$,
$$
h(D\phi(x))\leq -\min_{i\neq j} C_{ij}\alpha+C_{ji}\alpha
+\sum_{k\neq i,j}\min(C_{ik}\alpha,C_{jk}\alpha)=\alpha h(C).
$$
We apply Theorem~\ref{th-conratnonlsemi} and consider $y=\bold 1$ in the formula~\eqref{a-opconrate2}.
Since $\alpha\geq 0$ and $h(C)\leq 0$, we deduce that the set $U(w)$ is invariant. Therefore we conclude.
 \end{proof}
The Kuramoto equation~\cite{Strogatz00fromkuramoto} is a special case of the protocol~\eqref{a-nonlinearconsensus}.
\begin{align}\label{a-kuramoto}
\begin{array}{ll}
&\dot\theta_i=\displaystyle\sum_{j:\, (i,j)\in E}  C_{ij} \sin(\theta_j-\theta_i), i=1,\dots,n.\\
\end{array}
\end{align}
Let $w<\pi/2$. Then 
$$
\inf\{\cos(t):t\in [-w,w]\}\geq \cos w>0.
$$
We apply Proposition~\ref{pr-Uconvexopenbold1} to obtain that for all $\theta(0)$ such that $\othernorm{\theta(0)}_H<w$, the solution of~\eqref{a-kuramoto}
satisfies:
$$
\othernorm{\theta(t)}_H\leq e^ { h(C)\cos(w) t} \othernorm{\theta(0)}_H,\enspace \forall t\geq 0.
$$
In particular, for all $\theta(0) \in(-\pi/4,\pi/4)^ n$, the solution of equation~\eqref{a-kuramoto} satisfies:
$$
\othernorm{\theta(t)}_H\leq e^ {h(C)\cos(\othernorm{\theta(0)}_H)t} \othernorm{\theta(0)}_H,\enspace\forall t\geq 0.
$$

\begin{rem}
 Moreau~\cite{Moreau05} showed that if there is a node connected by path to all other nodes in the graph $(V,E)$, then the systems~\eqref{a-nonlinearconsensus} is globally 
convergent and~\eqref{a-kuramoto} is globally convergent on the set $(-\pi/2,\pi/2)^ n$. Compared to his results (see Remark~\ref{rem-compareMoreaulinear}), our condition for convergence is more strict but we 
obtain an explicit exponential contraction rate.
\end{rem}
Another class of maps satisfying~\eqref{a-gammaik} is $\gamma_{ik}(t)=\arctan(t)$. Consider the following system
\begin{align}\label{a-arctan}
\begin{array}{ll}
&\dot x_i=\displaystyle\sum_{j:\, (i,j)\in E}  C_{ij} \arctan(x_j-x_i), i=1,\dots,n.\\
\end{array}
\end{align}
Then we obtain in the same way that for all $x(0)\in \R^ n$, the solution of~\eqref{a-arctan} satisfies:
$$
\othernorm{x(t)}_H\leq e^ {\frac{h(C)}{1+x(0)^ 2}t}\othernorm{x(0)}_H,\enspace \forall t\geq 0.
$$
\end{example}

\begin{example}(Discrete $p$-Laplacian)
We now analyze the degenerate case of the $p$-Laplacian consensus dynamics for $p\in(1,2)\cup(2,+\infty)$. Then latter can be 
described by the dynamical system in $\R^ n$:
$$
\dot v_i=\sum_{j:\, (i,j)\in E} C_{ij}(v_j-v_i)
\left|C_{ij}(v_i-v_j)\right|^{p-2},  \enspace i=1,\dots,n.
$$
Let $\alpha>\beta>0$ and consider the convex open sets
$$
V(\beta):=\{v: \min_{i\neq j}|v_i-v_j|> \beta\},\quad U(\alpha):=\{v: \max_{i\neq j}|v_i-v_j|< \alpha\}.
$$
A basic calculus shows that for $v\in V(\beta)$,
$$
\frac{\partial \phi_i(v)}{\partial v_j}=\left\{\begin{array}{ll}
                                             0,& \enspace (i,j)\notin E\\
(p-1)|v_i-v_j|^ {p-2} C_{ij}^{p-1},&\enspace (i,j)\in E                                            
\end{array}
\right.
$$
Let $C^{p-1}$ denote the matrix with entries $C_{ij}^ {p-1}$.
Recall that $h(C^ {p-1})\leq 0$.
We have:
$$
h(D\phi(x))\leq 
\left\{\begin{array}{ll}
(p-1)h(C^{p-1})\beta^{p-2},&\enspace p>2,x\in V(\beta)\\
(p-1)h(C^{p-1})\alpha^{p-2},&\enspace 1<p<2,x\in V(\beta)\cap U(\alpha)
\end{array}\right.
$$
 When $1<p<2$, the contraction rate on $V(\beta)\cap U(\alpha)$ tends to $-\infty$ while $\alpha$ tends to 0.
When $p>2$, the contraction rate on $V(\beta)$ tends to $0$ while $\beta$ tends to $0$. If we fix some $\beta >\min_{(i,j)\in E} C_{ij}^ {-1}$, it can be 
checked that the contraction rate on $V(\beta)$ tends to $-\infty$ when $p$ tends to $+\infty$.
\end{example}

\section{Contraction rate in Hilbert's metric of non-linear flows}\label{sec-rateflowHilbertsmetric}
In this section, we apply Theorem~\ref{th-opnorm} to determine
the contraction rate in Hilbert's metric of the flow of an ordinary differential equations, still in the finite dimensional case.

\subsection{Contraction rate formula in Hilbert's metric}
In the following, we consider a continuously differentiable application $\phi:\C^ 0\rightarrow \cX$ such that $\phi(\lambda x)=\lambda \phi(x)$, for all $\lambda>0$ 
and $x\in \C^ 0$. Note that the later property implies that 
$$
D\phi(x)x=\phi(x).
$$
We denote by $M$ the flow associated to the differential equation (see Section~\ref{sec-nonlinearconsensus} for notations):
\begin{align}\label{a-xtphit}
 \dot x=\phi(x).
\end{align}
By uniqueness of the solution, it is clear that for all $x_0\in \C^ 0$,
$$
M_t(\lambda x_0)=\lambda M_t(x_0),\quad t\in J(x_0).
$$
Let $U\subset \C^0$ be a convex open set. Define the optimal contraction rate of the flow in Hilbert's metric on $U$ by:
\begin{align}\label{a-opconrate}
\begin{array}{l}
 \alpha(U):=
\inf\{\alpha\in \R: d_H(M_t(x_1),M_t(x_2))\leq e^{\alpha t}d_H(x_1,x_2), x_1,x_2\in U,
t\leq t_U(x_1)\wedge t_U(x_2)
\}.
\end{array}
\end{align}
For $x\in \C^ 0$, define:
\begin{align}\label{a-cx}
c(x):=-\inf_{ z\in [0, x]}\inf_{\substack{\nu,\pi \in \pP(x)\\ \<\pi,z>+\<\nu,x-z>=0}} \<\pi,D\phi(x)z>+\<\nu,D\phi(x)(x-z)>
\end{align}
For the same reason as in the proof of Proposition~\ref{pr-alphaHilsemiL}, it follows that
\begin{align}\label{a-cxextr}
c(x)=-\inf_{z\in \extr[0, x]}\inf_{\substack{\nu,\pi \in \extr\pP(x)\\  \<\pi,z>+\<\nu,x-z>=0}} \<\pi,D\phi(x)z>+\<\nu,D\phi(x)(x-z)>
\end{align}
\begin{theo}\label{th-Hilbertmetricnonlinearflow}
 Let $U\subset \C^ 0$ denote a convex open set such that $\lambda U=U$ for all $\lambda>0$. 
Then
\begin{align}\label{a-alphaUcx}
\alpha(U)=\sup_{x\in U} c(x).
\end{align}
\end{theo}
\begin{proof}
First we prove that for all $x\in U$,
$$
c(x)=\lim_{t\rightarrow 0^+} t^{-1}(\sup_{z} \frac{\hilbert{DM_t(x)z}{M_t(x)}}{\myhilbert(z/x)}-1).
$$
For this, fix $x\in U$ and define a functional on a neighborhood of $I$:
$$
F(W)=\sup_{z\in [0,x]}\sup_{\pi,\nu\in \pP(\unit)} \<\frac{\nu}{\<\nu,W(x)>}-\frac{\pi}{\<\nu,W(x)>},W(z)>.
$$
By Theorem~\ref{th-opnorm}, for $t\in[0,t_U(x))$,
$$
\begin{array}{ll}
 \othernorm{DM_t(x)}_H&=\displaystyle\sup_{z\in[0,x]}\sup_{\nu,\pi\in \pP(DM_t(x)x)} \<\nu-\pi,DM_t(x)z>\\
&=\displaystyle\sup_{z\in[0,x]}\sup_{\nu,\pi\in \pP(\unit)} \<\frac{\nu}{\<\nu,DM_t(x)x>}-\frac{\pi}{\<\pi,DM_t(x)x>},DM_t(x)z>\\
&=F(DM_t(x)).
\end{array}
$$
Therefore,
\begin{align}\label{a-DMtxz}
\begin{array}{ll}
&\displaystyle\lim_{t\rightarrow 0^+} t^{-1}(\sup\frac{\hilbert{DM_t(x)z}{M_t(x)}}{\myhilbert(z/x)}-1)\\
&=\displaystyle\lim_{t\rightarrow 0^+} t^{-1}(\othernorm{DM_t(x)}_H-1)\\
&=\displaystyle\lim_{t\rightarrow 0^+} t^{-1}(F(DM_t(x))-F(I))\\
\end{array}
\end{align}
Recall that
$DM_t(x):[0,t_U(x))\rightarrow \End(\cX)$ satisfies:
$$
\lim_{t\rightarrow 0^ +} t^ {-1}(DM_t(x)-I)=D\phi(x).
$$
The following reasoning is similar to that in the proof of~Proposition~\ref{pr-alphaHilsemiL}. 
The limit in~\eqref{a-DMtxz} equals to the semiderivative of $F$ at $I$ in the direction $D\phi(x)$, if this semiderivative exists. 
Since $[0,x]$ and $\pP(\unit)$
are compact sets and the function
$$
F_{\nu,\pi,z}(W)=\<\frac{\nu}{\<\nu,W(x)>}-\frac{\pi}{\<\pi,W(x)>},W(z)>
$$
is continuously differentiable on $W$ such that $F_{\nu,\pi,z}$ and the derivative $DF_{\nu,\pi,z}$ are jointly
continuous on $(\nu,\pi,z,W)$,
we know that $F$ is semidifferentiable.
The derivative of $F_{\nu,\pi,z}$ at point $I$ in the direction $D\phi(x)$ is:
$$
\begin{array}{ll}
&DF_{\nu,\pi,z}(I)(D\phi(x))\\
&=\displaystyle\frac{\<\nu,D\phi(x)z>\<\nu,x>-\<\nu,z>\<\nu,D\phi(x)x>}{\<\nu,x>^ 2}-\frac{\<\pi,D\phi(x)z>\<\pi,x>-\<\pi,z>\<\pi,D\phi(x)x>}{\<\pi,x>^ 2}\\
&=\<\frac{\nu}{\<\nu,x>},D\phi(x)z>-\<\frac{\nu}{\<\nu,x>},z>\<\frac{\nu}{\<\nu,x>},D\phi(x)x>-\<\frac{\pi}{\<\pi,x>},D\phi(x)z>+\<\frac{\pi}{\<\pi,x>},z>\<\frac{\pi}{\<\pi,x>},D\phi(x)x>
\end{array}
$$
Denote
$$
T(W)=\argmax{\substack{\nu,\pi\in\pP(\unit)\\z\in[0,x]}} F_{\nu,\pi,z}(W).
$$
Then 
$$
T(I)=\{\nu,\pi\in \pP(\unit),z\in [0,x]:\<\frac{\nu}{\<\nu,x>}-\frac{\pi}{\<\pi,x>},z>=1\}.
$$
The semiderivative of $F$ at point $I$ in the direction $D\phi(x)$ is then:
$$
\begin{array}{ll}
&\displaystyle\lim_{t\rightarrow 0^+} t^{-1}(F(DM_t(x))-F(I))\\
&=\displaystyle\sup_{\nu,\pi,z\in T(I)}  DF_{\nu,\pi,z}(W)(D\phi(x))\\
&=\displaystyle\sup_{z\in[0,x]}\sup_{\substack{\nu,\pi \in \pP(x)\\ \<\nu-\pi,z>=1}} \<\nu,D\phi(x)z>-\<\nu,D\phi(x)x>-\<\pi,D\phi(x)z>\\
&=c(x).
\end{array}
$$
Now fix $x_0\in U$. By Cauchy-Lipschitz, there is $r>0$ and $t_0>0$ such that the flow is well-defined on $[0,t_0]\times B(x_0;r)$ where $B(x_0;r)$
is the open ball of radius $r$ centered at $x_0$. We assume that $B(x_0;r)\subset U$ and consider the set $G:=\cup_{\lambda>0}\lambda B(x_0;r)$.
For every $t\leq t_0$, the application $M_t$ is well defined on $G$ such that 
$$
d_H(M_t(x),M_t(y))\leq e^{\alpha(U)t} d_H(x,y),\enspace \forall x,y\in G.
$$
By Theorem~\ref{th-Nus94}, we have
$$
\myhilbert(DM_t(x)v/M_t(x))\leq e^{\alpha(U)t}\myhilbert(v/x)\enspace \forall x\in G, v\in \cX.
$$
Therefore,
$$
c(x_0)=\limsup_{t\rightarrow 0^+}  \frac{1}{t}(\sup_{z} \frac{\myhilbert(DM_t(x_0)z/M_t(x_0))}{\myhilbert(z/x_0)}-1) \leq \alpha(U).
$$
It follows that $$\alpha(U)\geq \sup_{x\in U} c(x).$$
Finally, denote $$c=\sup_{x\in U} c(x).$$ 
Then for all $x\in U$, $v\in \cX$ and $t\in t_U(x)$,
$$
\begin{array}{ll}
&\displaystyle\limsup_{h\rightarrow 0^+} \frac{\myhilbert(DM_{t+h}(x)v/M_{t+h}(x))-\myhilbert(DM_t(x)v/M_t(x))}{h}\\
&=\displaystyle\limsup_{h\rightarrow 0^+} \frac{\myhilbert(DM_{h}(M_t(x))(DM_t(x)v))/M_{h}(M_t(x)))-\myhilbert(DM_t(x)v/M_t(x))}{h}\\
&=\displaystyle\limsup_{h\rightarrow 0^+} \frac{\myhilbert(DM_t(x)v/M_t(x))}{h} (\frac{\myhilbert(DM_{h}(M_t(x))(DM_t(x)v))/M_{h}(M_t(x))}{\myhilbert(DM_t(x)v/M_t(x))}-1)\\
&\leq c(M_t(x))\myhilbert(DM_t(x)v/M_t(x)) \leq c \myhilbert(DM_t(x)v/M_t(x)).
\end{array}
$$
Therefore,  for all $x\in U$, $v\in \cX$ and $t\in t_U(x)$ we have that,
$$
\myhilbert(DM_t(x)v/M_t(x))\leq e^{ct} \myhilbert(v/x).
$$
 Let $x,y\in U$ and define $\gamma(s)=(1-s)x+sy$, $0\leq s\leq 1$. By the compacity of the set $\{\gamma(s):s\in[0,1]\}$, we know that $$t_0:=\inf\{t_U(\gamma(s)):s\in [0,1]\}>0.$$
Therefore, using the Finsler structure of Hilbert's metric (\cite[Thm 2.1]{nussbaum94}), we get that for every $t\leq t_0$,
$$
\begin{array}{ll}
&d_H(M_t(x),M_t(y))\leq \int_{0}^1 \myhilbert(DM_t(\gamma(s))(y-x)/M_t(\gamma(s))) ds\\
&\leq \int_{0}^1 e^{ct} \myhilbert(y-x/\gamma(s))ds\\
&=e^{ct} d_H(x,y).
\end{array}
$$
Consequently we proved that for all $x,y\in U$
$$
\limsup_{h\rightarrow 0^+} \frac{d_H(M_h(x),M_h(y))-d_H(x,y)}{h}\leq cd_H(x,y).
$$
This implies that for all $x,y\in U$ and $t< t_U(x)\wedge t_U(y)$:
$$
\begin{array}{ll}
&\displaystyle\limsup_{h\rightarrow 0^+} \frac{d_H(M_{t+h}(x),M_{t+h}(y))-d_H(M_t(x),M_t(y))}{h}\\
&=\displaystyle\limsup_{h\rightarrow 0^+} \frac{d_H(M_{h}(M_t(x)),M_{h}(M_t(y)))-d_H(M_t(x),M_t(y))}{h}\\
&\leq cd_H(M_t(x),M_t(y)).
\end{array}
$$
It follows that
$$
d_H(M_t(x),M_t(y))\leq e^{ct}d_H(x,y),\enspace \forall x,y\in U, t< t_U(x)\wedge t_U(y).
$$
Therefore $$\alpha(U)\leq c.$$
\end{proof}

\subsection{Contraction rate in Hilbert's projective metric of a non-linear flow on the standard positive cone}
We specialize the contraction formula~\eqref{a-alphaUcx} to the case $\cX=\R$ and $\C=\R^+$ under the same notations and assumptions.
\begin{coro}\label{coro-conrateHilRn}
When $\cX=\R^n$ and $\C=\R_n^+$, the contraction rate formula~\eqref{a-alphaUcx} can be specified as below:
$$
\alpha(U)=\sup_{x\in U}c(x)=\sup_{x\in U} h(A(x)),\enspace \forall x\in U
$$
where
$$
A(x)=\delta(x)^{-1}D\phi(x)\delta(x)
$$
and $h$ is defined in~\eqref{a-hA}.
\end{coro}
\begin{proof}
It is sufficient to remark that in this special case:
$$
\extr \pP(x)=\delta(x)^ {-1}\extr\pP,
$$
and
$$
\extr[0,x]=\delta(x) \extr([0,\bold 1])
$$
Therefore,
$$
\begin{array}{ll}
c(x)&=-\displaystyle\inf_{z\in \extr[0, x]}\inf_{\substack{\\\pi,\nu \in \extr\pP(x)\\  \<\nu,z>+\<\pi,x-z>=0}} \<\nu,D\phi(x)z>+\<\pi,D\phi(x)(x-z)>\\
&=-\displaystyle\inf_{z\in \extr[0, \bold 1]}\inf_{\substack{\\\pi,\nu \in \extr\pP\\  \<\nu,z>+\<\pi,x-z>=0}} \<\delta(x)^ {-1}\nu,D\phi(x)\delta(x)z>+\<\delta(x)^ {-1}\pi,D\phi(x)\delta(x)(\bold 1-z)>\\
&=h(A(x)).
\end{array}
$$
\end{proof}
\begin{rem}
 Consider the linear flow in $\R^ n$ of the following equation:
$$
\dot x=Ax,
$$
where $A_{ij}\geq 0$, for all $i\neq j$, so that the flow is order preserving.
Let $x$ be in the interior of $\R^ n_+$. Then we have
$$
\delta(x)^ {-1}A\delta(x)_{ij}=A_{ij}\frac{x_j}{x_i},\enspace i,j=1,\dots,n.
$$ 
Therefore, 
$$
h(\delta(x)^ {-1}A\delta(x))=-\min_{i\neq j} A_{ji}\frac{x_i}{x_j}+A_{ij}\frac{x_j}{x_i}+\sum_{k\notin\{i,j\}} \min(A_{ik}\frac{x_k}{x_i},A_{jk}\frac{x_k}{x_j}).
$$
The global contraction rate (restricted to $\C^ 0$) is then
$$
\sup_{x\in \C^ 0} h(\delta(x)^ {-1}A\delta(x))=-\min_{i\neq j} 2\sqrt{A_{ij}A_{ji}}.
$$
Such contraction rate can be alternatively obtained by differentiating  with respect to $t$ at $0$ the contraction ratio of $I+tA$, using Birkhoff's theorem.
Hence a positive global contraction rate exists if and only if $A_{ij}>0$ for all $i\neq j$. However, strict local contraction may
occur even if there is $A_{ij}=0$ for some $i\neq j$. Let $K>1$ and consider the convex open set 
$$
U(K)=\{x\in \R^ n: \frac{1}{K}\leq \frac{x_i}{x_j}\leq K\}.
$$
Then the local contraction rate with respect to $U$ is
$$
\sup_{x\in U(K)} h(\delta(x)^ {-1}A\delta(x)) \leq \frac{h(A)}{K}.
$$
Therefore, $h(A)<0$ is sufficient to have a strict local contraction. Moreover, the above bound on the contraction rate decreases (faster convergence) as the 
orbit approaches to consensus, i.e., a multiple of $\bold 1$.
\end{rem}

\subsection{Application to the space of Hermitian matrices}
We specialize the contraction formula~\eqref{a-alphaUcx} to the case $\cX=\sym_n$ and $\C=\sym_n^+$ under the same notations and assumptions. 
 \begin{coro}\label{coro-applicationtoSn}
  When $\cX=\sym_n$ and $\C=\sym_n^ +$, the contraction rate formula~\eqref{a-alphaUcx} can be specified as below:
$$
\alpha (U)=\sup_{P\in U} c(P)=\sup_{P\in U} h(\Phi(P))
$$
where
$
\Phi(P):\sym_n^ +\rightarrow \sym_n^ +
$ is a linear application given by:
$$
\Phi(P)(Z)=P^{-\frac{1}{2}}D\phi(P)(P^{\frac{1}{2}}ZP^{\frac{1}{2}})P^{-\frac{1}{2}}
$$
and $h$ is defined in~\eqref{a-hPhi}.
 \end{coro}
\begin{proof}
 Remark that in this special case,
$$
\extr[0,P]=P^ {\frac{1}{2}}(\extr[0,I_n])P^{\frac{1}{2}},
$$
and
$$
\extr(\pP(P))=P^ {-\frac{1}{2}}(\extr\pP)P^{-\frac{1}{2}}. 
$$
The desired formula is obtained the same way as in the proof of Corollary~\ref{coro-conrateHilRn}.
\end{proof}
\begin{example}
 As an example, let us show a calculus of contraction rate using Corollary~\ref{coro-applicationtoSn} for the following differential equation
in $\sym_n$:
\begin{align}\label{aPBP}
\dot P=\phi(P):=\frac{-PBP}{\trace(CP)} +AP+PA'
\end{align}
where $B,C\in \hat\sym_n^+$. Let $\hat P\in \sym_n^ +$. Then the application $\Phi(P):\sym_n\rightarrow \sym_n$ defined in Corollary~\ref{coro-conrateHilRn}
is given by:
$$
\begin{array}{ll}
\Phi(P)(Z)&=P^{-\frac{1}{2}}D\phi(P)(P^{\frac{1}{2}}ZP^{\frac{1}{2}})P^{-\frac{1}{2}}\\
&= (-ZP^ {\frac{1}{2}} BP^ {\frac{1}{2}}-P^ {\frac{1}{2}}BP^ {\frac{1}{2}} Z)\trace(CP)^{-1}\\
&~~~~~\quad~+
P^ {\frac{1}{2}}BP^ {\frac{1}{2}}\trace(CP)^ {-2}\trace(CP^ {\frac{1}{2}}ZP^{\frac{1}{2}})+P^ {-\frac{1}{2}}AP^ {\frac{1}{2}}Z
+ZA'P^ {-\frac{1}{2}}
\end{array}
$$
Therefore let $x,y\in \CC^ n$ such that $x^ *y=0$ then
$$
y^*\Phi(P)(xx^*)y=y^*P^ {\frac{1}{2}}BP^ {\frac{1}{2}}y\trace(CP^ {\frac{1}{2}}xx^*P^{\frac{1}{2}})\trace(CP)^ {-2}.
$$
Let $\{x_1,\dots,x_n\}$ be an orthonormal basis. Denote $\alpha_1=x_1^*P^{\frac{1}{2}}BP^ {\frac{1}{2}}x_1$, 
$\alpha_2=x_2^*P^{\frac{1}{2}}BP^ {\frac{1}{2}}x_2$, $\beta_1=x_1^*P^{\frac{1}{2}}CP^ {\frac{1}{2}}x_1$ and
$\beta_2=x_1^*P^{\frac{1}{2}}CP^ {\frac{1}{2}}x_1$. Suppose that $\alpha_1\leq \alpha_2$.
Then 
$$
\begin{array}{ll}
&x_1^*\Phi(P)(x_2x_2^ *)x_1 +x_2^*\Phi(P)(x_1x_1^ *)x_2+\displaystyle\sum_{k=3}^ n \min(x_1^*\Phi(P)(x_kx_k^ *)x_1 , x_2^*\Phi(P)(x_kx_k^ *)x_2)\\
&=\big(\alpha_1\beta_2+\alpha_2\beta_1+\alpha_1\displaystyle\sum_{k=3}^ n x_k^ *P^ {\frac{1}{2}}CP^ {\frac{1}{2}}x_k\big)\trace(CP)^ {-2} \\
&=\big(\alpha_1\beta_2+\alpha_2\beta_1+\alpha_1(\trace(CP)-\beta_1-\beta_2))\trace(CP)^ {-2}\\
&=\big(\alpha_1\trace(CP)+\beta_1(\alpha_2-\alpha_1))\trace(CP)^ {-2}\\
&\geq \lambda_{\min}(BP)\trace(CP)^{-1}.
\end{array}
$$
Therefore by the definition in~\eqref{a-hPhi},
$$
h(\Phi(P))\leq -\lambda_{\min}(BP)\trace(CP)^ {-1}.
$$
Let us consider the convex open set 
$$
U=\{P\in \hat \sym_n^ +:d_H(P,I_n)< K\}.
$$
Then,
$$
\begin{array}{ll}
\displaystyle\sup_{P\in U} h(\Phi(P))&\leq \sup_{P\in U} -\lambda_{\min}(BP)\trace(CP)^ {-1}\\
& \leq
-\displaystyle\frac{\lambda_{\min}(BP)}{n\lambda_{\max}(CP)}\leq
 -\displaystyle\frac{\lambda_{\min}(B)\lambda_{\min}(P)}{n\lambda_{\max}(C)\lambda_{\max}(P)} \\
&\leq -\displaystyle\frac{\lambda_{\min}(B)}{n\lambda_{\max}(C)e^{K}}
\end{array}
$$
Let $\alpha=-\displaystyle\frac{\lambda_{\min}(B)}{n\lambda_{\max}(C)e^{K}}$.
Then by Corollary~\ref{coro-applicationtoSn}, for all $P_1,P_2\in U$ we have:
$$
d_H(M_t(P_1),M_t(P_2))\leq e^{\alpha t}d_H(P_1,P_2),\enspace 0\leq t<t_{U}(P_1)\wedge t_{U}(P_2).
$$
If $A,B,C$ are matrices such that 
$$\phi(I_n)=-B\trace(C)^ {-1}+A+A'=\lambda_0 I_n,$$ then we know that
$$M_t(I_n)=e^ {\lambda_0 t}I_n.$$
In that case, for $P\in U$ we have:
$$
d_H(M_t(P),e^ {\lambda_0 t}I_n)\leq e^{\alpha t}d_H(P,I_n),\enspace 0\leq t<t_{U}(P).
$$
It follows that $t_U(P)=+\infty$ and therefore every solution of equation~\eqref{aPBP} converges exponentially to a scalar multiplication of 
$I_n$.
\end{example}

\bibliographystyle{alpha}
\bibliography{biblio}
\end{document}